\documentclass[12pt,a4paper]{article}

\usepackage[english]{babel}

\usepackage[a4paper,top=2cm,bottom=2cm,left=3cm,right=3cm,marginparwidth=1.75cm]{geometry}
	\usepackage{authblk}

\usepackage{amsmath}
\usepackage{amssymb}
\usepackage{amsthm}
\numberwithin{equation}{section}
\numberwithin{figure}{section}
\numberwithin{table}{section}
\usepackage{graphicx}
\usepackage[colorlinks=true, allcolors=blue]{hyperref}
\usepackage{booktabs}
\usepackage{multirow}
\usepackage{subcaption}

\newtheorem{theorem}{Theorem}[section]
\newtheorem{remark}{Remark}[section]

\newcounter{example}[section]
\renewcommand{\theexample}{\thesection.\arabic{example}}

\newcommand{\exampletitle}[2]{%
  \refstepcounter{example}%
  \label{#1}%
  \subsubsection*{Example~\theexample. #2}%
}

\title{A well-balanced path-conservative discontinuous Galerkin method with generalized hydrostatic reconstruction for non-conservative balance laws}
\author[1,2]{Ruilin Fan}
\author[1]{Yinhua Xia}
\author[1,2,$*$]{Yan Xu}
\affil[1]{School of Mathematical Sciences, University of Science and Technology of China}
\affil[2]{Laoshan Laboratory}
\affil[$*$]{Corresponding author, Email: yxu@ustc.edu.cn}

\date{}

\begin{document}

\maketitle

\begin{abstract}
We propose a high-order well-balanced path-conservative discontinuous Galerkin method for one-dimensional balance laws with non-conservative products. Generalized hydrostatic reconstruction at cell interfaces couples effective numerical fluxes with Dal Maso-LeFloch-Murat path integrals, while a cellwise reference-equilibrium decomposition cancels projection-induced residuals in the volume discretization. Under explicit assumptions on equilibrium recovery, admissibility, paths, and boundary data, the spatial discretization is path conservative, and prescribed rest and moving equilibria are exactly preserved by both the semi-discrete scheme and its fully discrete SSP-RK approximation. The framework is applied to the two-layer shallow water equations and a blood-flow model with spatially varying vessel properties. Numerical experiments confirm the designed order of accuracy and round-off-level preservation of smooth and discontinuous equilibria, demonstrate accurate resolution of small perturbations, and show that the interface treatment captures shocks and stationary material contacts without visible spurious oscillations.

\end{abstract}

\section{Introduction}

Hyperbolic balance laws are conservative when their spatial operator can be written as the divergence of a flux, up to prescribed source terms. Their integral form then determines the interface flux contribution, and discontinuities satisfy the classical Rankine--Hugoniot condition. Non-conservative balance laws contain additional products, typically of the form $G(\boldsymbol{u})\boldsymbol{u}_x$, that cannot in general be expressed as the derivative of a flux depending only on the state. Such products arise naturally in reduced, averaged, or coupled models where different phases, layers, materials, or closure variables interact. Examples include compressible two-medium and multiphase flows \cite{RHEBERGEN2008,cheng2022bound,KEMM2020104536}, multi-layer geophysical and two-layer shallow water systems \cite{kurganov2009central,hernandez2021central,ZHANG2025}, and shallow-flow moment models \cite{koellermeier2020analysis,KOELLERMEIER2022127166,fan2026well}. Because the non-conservative contribution is not determined by a conservative flux alone, jumps and discontinuities require additional structure beyond the standard conservative formulation.

Across discontinuities, a non-conservative product is not defined in the classical distributional sense, and the usual Rankine--Hugoniot condition does not provide a complete weak formulation. The theory of Dal Maso, LeFloch, and Murat (DLM) assigns a weak meaning to such products through a prescribed family of paths in the phase space \cite{Maso1995}. For each jump, the path specifies how the two states are connected in phase space, and the non-conservative contribution is defined by integration along this path. The path is therefore part of the weak formulation rather than a purely numerical choice. This path-dependent weak interpretation provides the theoretical basis for path-conservative approximations.

Path-conservative schemes generalize conservative methods by matching the numerical interface fluctuations to a prescribed DLM path family and its non-conservative path integral \cite{pares2006numerical}. This ties the numerical treatment of jumps to the selected weak formulation. In finite volume methods, Castro et al. developed high-order reconstruction-based schemes for general non-conservative systems \cite{castro2006high}. Central-upwind ideas were adapted to two-layer shallow water models and later cast as a second-order path-conservative method \cite{kurganov2009central,castro2019path}. Fifth-order alternative weighted essentially non-oscillatory (A-WENO) extensions followed \cite{CHU2022111508}. In the discontinuous Galerkin (DG) setting, Rhebergen et al. incorporated the DLM path contribution into space and space-time formulations \cite{RHEBERGEN2008}. Subsequent path-conservative DG schemes have been applied to compressible two-medium flows and two-layer shallow water equations \cite{cheng2022bound,ZHANG2025}. Guerrero Fern\'andez et al. developed a well-balanced ADER-DG method for a variable-density multilayer shallow water model \cite{guerrero2022aderdg}. Other high-order developments include arbitrary high-order derivative (ADER) schemes, path-conservative ADER-DG methods, ADER-WENO finite volume schemes, and entropy-stable DG spectral element methods \cite{DUMBSER20091731,KEMM2020104536,DUMBSER2014359,RENAC20191}. These methods systematically discretize path-dependent interface contributions. Nevertheless, path consistency alone does not guarantee convergence. Formally consistent schemes may still fail to recover the shock solution selected by the prescribed path or the underlying regularization \cite{hou1994nonconservative,CASTRO20088107,ABGRALL20102759}.

Well-balanced schemes are designed to preserve selected steady states exactly at the discrete level. Many balance laws admit nontrivial equilibria in which flux gradients are balanced by source terms and, for non-conservative systems, by non-conservative products. If these contributions are discretized independently, the resulting imbalance can generate spurious waves near steady flows, obscure small perturbations, or accumulate during long-time simulations. Well-balanced methods prevent this behavior by incorporating the equilibrium relation into the numerical fluxes, source discretization, and, when necessary, path-integral contributions. Foundational source treatments were developed by Greenberg and LeRoux, Bermudez and Vazquez, and LeVeque \cite{dgreenberg1996well,BERMUDEZ19941049,LEVEQUE1998346}. For the shallow water equations, hydrostatic reconstruction was introduced for lake-at-rest equilibria and later generalized to broader stationary states \cite{Audusse2004AFA,CASTRO2007}. Castro D\'iaz et al. combined GHR with high-order reconstruction to preserve every stationary shallow water solution \cite{CASTRODIAZ2013242}. Xing and Shu developed high-order well-balanced WENO and DG methods for systems with source terms \cite{XING2006567}. For moving-water equilibria, Noelle et al. developed high-order WENO finite volume schemes, and Xing developed an exactly well-balanced DG method \cite{NOELLE200729,XING2014536}. Britton and Xing developed well-balanced DG methods for blood-at-rest and moving-flow equilibria \cite{BRITTON2020104493}. Related constructions cover non-explicit steady states in the Ripa model \cite{desveaux2016well} and hydrostatic or moving equilibria in the Euler equations with gravity \cite{chandrashekar2015second,LI2016145,klingenberg2019arbitrary}. These developments establish steady-state preservation as a structural requirement distinct from path consistency. In high-order DG or path-conservative DG (PCDG) schemes, exact preservation further requires consistent treatment of both interface and cell-interior contributions.

Path-conservative and well-balanced schemes address different requirements: the former discretize path-dependent jumps, whereas the latter enforce steady-state cancellation. Recent work has begun to combine both goals. Equilibrium-preserving-space DG methods have been developed for moving-water equilibria in the shallow water equations and for general hyperbolic balance laws \cite{Zhang2023,zhang2024}. They have also been extended to high-order well-balanced path-conservative schemes for specific non-conservative shallow-water-type systems \cite{ZHANG2025,fan2026well}. However, using equilibrium variables as polynomial unknowns introduces local nonlinear stage solves and increases dependence on the model-specific equilibrium map. Flux-globalization central-upwind schemes are both well balanced and path conservative \cite{KURGANOV2023111773}; their original formulation is second order, retains central-upwind dissipation, and requires global-flux integration. A fifth-order A-WENO path-conservative central-upwind extension preserves lake-at-rest equilibria but does not provide high-order flux globalization for the more general equilibria considered here \cite{CHU2022111508}.

For high-order PCDG schemes, equilibrium imbalance arises at both interfaces and in cell interiors. Interface terms contain numerical fluxes and DLM path integrals, whereas interior terms are affected by polynomial projection and the discretization of sources and non-conservative products. GHR adjusts interface states while preserving specified equilibrium variables \cite{CASTRO2007}, but it does not by itself remove the projection-induced imbalance inside DG cells. Conversely, the reference-equilibrium methods of Noelle et al. and Xing recover a cellwise equilibrium from projected data and evaluate the residual relative to that state \cite{NOELLE200729,XING2014536}. This observation motivates combining GHR with a reference-equilibrium decomposition in a single path-conservative DG formulation.

This paper extends the conservative-variable reference-equilibrium/GHR strategy of Xing \cite{XING2014536} to balance laws with non-conservative products. The proposed PCDG-GHR method uses equilibrium variables to reconstruct interface states and couples the effective numerical fluxes to DLM path integrals. In each cell, a locally recovered reference equilibrium defines a perturbation residual that treats the conservative flux, non-conservative product, and source term relative to the same equilibrium state. We prove path conservation and exact well balancing under explicit recovery, path, admissibility, interface-matching, and boundary assumptions. Numerical experiments for the two-layer shallow water equations and a blood-flow model with heterogeneous vessel properties assess high-order accuracy, equilibrium preservation, small-perturbation resolution, and discontinuous-flow performance.

The remainder of this paper is organized as follows. Section~\ref{se:sys} presents the balance-law framework, DLM path setting, and standard semi-discrete PCDG formulation. Section~\ref{se:ghr} develops the PCDG-GHR scheme, including the GHR interface treatment, reference equilibrium interior residual, fully discrete scheme, and limiter compatibility. Section~\ref{sec:exactly-well-balanced-property} proves path conservation of the GHR interface fluctuations and exact well balancing at the semi-discrete and fully discrete levels. Section~\ref{se:applications} specifies the two model realizations. Section~\ref{se:nu} reports numerical results for the two-layer shallow water equations and the blood flow equations. Section~\ref{se:con} concludes the paper.

\section{Path-conservative DG framework}\label{se:sys}

We introduce the general balance-law framework and the equilibrium variables used in the method. We then describe the DLM path setting and the standard semi-discrete PCDG formulation.

\subsection{Non-conservative systems and equilibrium states}\label{subsec:general-form}

We consider a class of one-dimensional hyperbolic balance laws with non-conservative products,
\begin{equation}
\boldsymbol{u}_t + \boldsymbol{f}(\boldsymbol{u})_x + G(\boldsymbol{u})\boldsymbol{u}_x = \boldsymbol{s}(\boldsymbol{u},\boldsymbol{\sigma}),
\label{eq:nonconservative-system}
\end{equation}
where $\boldsymbol{u} \in \Omega \subset \mathbb{R}^m$ is the $m$-component state vector used in the DG spatial operator and $\Omega$ is the set of admissible states. The prescribed datum $\boldsymbol{\sigma}=\boldsymbol{\sigma}(x)$ may be scalar or vector valued; it may be supplied separately from $\boldsymbol{u}$ or included as stationary components of $\boldsymbol{u}$. In the latter case, those components are held at their prescribed values rather than updated as independent unknowns. The function $\boldsymbol{f}(\boldsymbol{u})$ is the conservative flux, $G(\boldsymbol{u})\boldsymbol{u}_x$ represents the non-conservative product, and $\boldsymbol{s}(\boldsymbol{u},\boldsymbol{\sigma})$ is the cellwise source. As in the model notation below, dependence of this source on prescribed spatial derivatives is implicit; these derivatives are evaluated within each cell. A source contribution at a jump of an external datum requires an augmented interface path; if the datum is already a stationary block of $\boldsymbol{u}$, its spatial derivative is included in $G(\boldsymbol{u})\boldsymbol{u}_x$ and is not counted again in $\boldsymbol{s}$. Conservative balance laws are included as the special case $G(\boldsymbol{u})\equiv 0$. For later use, we introduce the quasilinear matrix
\begin{equation}
\mathcal{A}(\boldsymbol{u}) = \frac{\partial \boldsymbol{f}}{\partial \boldsymbol{u}}(\boldsymbol{u}) + G(\boldsymbol{u}),
\label{eq:quasilinear-matrix}
\end{equation}
so that \eqref{eq:nonconservative-system} may be rewritten as
\begin{equation}
\boldsymbol{u}_t + \mathcal{A}(\boldsymbol{u})\boldsymbol{u}_x = \boldsymbol{s}(\boldsymbol{u},\boldsymbol{\sigma}).
\label{eq:quasilinear-system}
\end{equation}

Following \cite{CASTRO2007}, when the prescribed datum is external to $\boldsymbol{u}$, let $-C(\boldsymbol{u},\boldsymbol{\sigma})\boldsymbol{\sigma}_x$ denote its derivative contribution to the source. The augmented state vector and augmented matrix for interface jumps are
\begin{equation}
\boldsymbol{w}=\begin{pmatrix}\boldsymbol{u}\\\boldsymbol{\sigma}\end{pmatrix},
\qquad
\widetilde{\mathcal{A}}(\boldsymbol{w})=
\begin{pmatrix}
\mathcal{A}(\boldsymbol{u}) & C(\boldsymbol{u},\boldsymbol{\sigma})\\
0 & 0
\end{pmatrix}.
\label{eq:augmented-state-matrix}
\end{equation}
Here $C$ is the coefficient matrix of the prescribed-data derivative, and the zero rows express $\boldsymbol{\sigma}_t=0$. When $\boldsymbol{s}=-C(\boldsymbol{u},\boldsymbol{\sigma})\boldsymbol{\sigma}_x$, this construction absorbs the source into the augmented matrix and yields $\boldsymbol{w}_t+\widetilde{\mathcal{A}}(\boldsymbol{w})\boldsymbol{w}_x=0$, with zero source in the augmented system. This formulation follows \cite{CASTRO2007} and facilitates the subsequent path-conservation proof; the numerical scheme itself retains the original balance-law form with $\boldsymbol{s}(\boldsymbol{u},\boldsymbol{\sigma})$ and does not require an augmented formulation. If the prescribed components are already included in $\boldsymbol{u}$, they are not appended again: $\boldsymbol{w}=\boldsymbol{u}$, $\widetilde{\mathcal{A}}=\mathcal{A}$, and the additional zero block in the interface fluctuations is omitted.

Throughout, we assume that $\boldsymbol{f}$ is continuously differentiable on admissible states and that $\widetilde{\mathcal{A}}$ is continuous along the admissible paths considered below.

A stationary solution of \eqref{eq:nonconservative-system} satisfies
\begin{equation}
\boldsymbol{f}(\boldsymbol{u})_x + G(\boldsymbol{u})\boldsymbol{u}_x = \boldsymbol{s}(\boldsymbol{u},\boldsymbol{\sigma}).
\label{eq:steady-balance}
\end{equation}
In this paper, an equilibrium state means a stationary solution satisfying the steady balance relation \eqref{eq:steady-balance}. In general, this relation admits nontrivial equilibrium states rather than only constant equilibria. Preserving such states at the discrete level is the defining feature of a well-balanced numerical method.

The framework is intended for models whose target equilibria can be described by explicitly characterized equilibrium variables. We denote this vector of equilibrium variables by $\boldsymbol{v}$. Its components remain constant along the selected equilibrium state and are model-dependent; they may include energy-, enthalpy-, or Bernoulli-type invariants together with transport components such as discharge or momentum.

The construction below assumes that the state variables and the equilibrium variables can be related, for the selected target equilibrium family, by a model-dependent transformation
\begin{equation}
\boldsymbol{v}=\tau(\boldsymbol{u},\boldsymbol{\sigma}),
\qquad
\boldsymbol{u}=\tau^{-1}(\boldsymbol{v},\boldsymbol{\sigma})\in\Omega.
\label{eq:equilibrium-variable-transform}
\end{equation}
Here $\tau^{-1}$ denotes a local inverse on a selected branch. The existence of global equilibrium quantities is therefore not, by itself, sufficient for the method to be applied. One must first specify the target equilibrium family and select a locally recoverable branch of the map from the equilibrium variables and prescribed data back to a state in $\Omega$. In this paper, an admissible branch refers to such a locally selected inverse branch that satisfies the physical constraints encoded by the admissible state set and remains consistent with the selected flow state or flow branch. Continuity from the preceding time level or Runge--Kutta stage may be used when more than one admissible inverse is available.

\subsection{DLM theory and path definition}\label{subsec:dlm-path}

For discontinuous solutions, the non-conservative product $G(\boldsymbol{u})\boldsymbol{u}_x$ cannot be interpreted in the classical distributional sense. To give it a rigorous weak meaning, we follow the theory of Dal Maso, LeFloch, and Murat (DLM) \cite{Maso1995}. We first consider the homogeneous jump operator associated with \eqref{eq:quasilinear-system},
\begin{equation}
\boldsymbol{u}_t + \mathcal{A}(\boldsymbol{u})\boldsymbol{u}_x = 0,
\label{eq:homogeneous-quasilinear-system}
\end{equation}
where $\mathcal{A}(\boldsymbol{u})=\partial \boldsymbol{f}/\partial \boldsymbol{u} + G(\boldsymbol{u})$. This operator describes jumps of $\boldsymbol{u}$ at fixed external data and also applies when stationary data are already components of $\boldsymbol{u}$; it does not by itself account for jumps of data kept outside $\boldsymbol{u}$. In the DLM framework, the non-conservative product is defined through a family of paths in the phase space $\Omega$. Following Par\'{e}s \cite{pares2006numerical}, we consider a locally Lipschitz path family
\begin{equation}
\boldsymbol{\Phi} : [0,1]\times\Omega\times\Omega \to \Omega,
\label{eq:path-family}
\end{equation}
such that, for any pair of states $\boldsymbol{u}^-,\boldsymbol{u}^+\in\Omega$,
\begin{equation}
\boldsymbol{\Phi}(0;\boldsymbol{u}^-,\boldsymbol{u}^+) = \boldsymbol{u}^-,
\qquad
\boldsymbol{\Phi}(1;\boldsymbol{u}^-,\boldsymbol{u}^+) = \boldsymbol{u}^+.
\label{eq:path-endpoints}
\end{equation}
The path family also satisfies the consistency condition
\begin{equation}
\boldsymbol{\Phi}(s;\boldsymbol{u},\boldsymbol{u})=\boldsymbol{u},
\qquad
s\in[0,1],
\label{eq:path-consistency}
\end{equation}
so that every path integral with coincident endpoints vanishes.
Given such a path, the jump contribution of the non-conservative term across an interface is defined by
\begin{equation}
\int_0^1 G(\boldsymbol{\Phi}(s;\boldsymbol{u}^-,\boldsymbol{u}^+))\boldsymbol{\Phi}_s(s;\boldsymbol{u}^-,\boldsymbol{u}^+) \, ds.
\label{eq:nonconservative-jump}
\end{equation}
Accordingly, if a discontinuity propagates with speed $\xi$, then the associated weak solution must satisfy the generalized Rankine--Hugoniot condition
\begin{equation}
\int_0^1 \bigl(\xi I-\mathcal{A}(\boldsymbol{\Phi}(s;\boldsymbol{u}^-,\boldsymbol{u}^+))\bigr)\boldsymbol{\Phi}_s(s;\boldsymbol{u}^-,\boldsymbol{u}^+) \, ds = 0,
\label{eq:generalized-rankine-hugoniot}
\end{equation}
where $I$ denotes the identity matrix. This relation reduces to the classical Rankine--Hugoniot condition in the conservative case. It also shows that, for non-conservative systems, the choice of the path family is part of the definition of weak solutions, and therefore it directly influences the design of path-conservative numerical methods.

\subsection{Standard semi-discrete PCDG formulation}

Let $\{I_j\}$ be a partition of the computational domain, and let $V_h^k$ denote the discontinuous Galerkin space of piecewise polynomials of degree at most $k$ on this partition. A standard semi-discrete path-conservative discontinuous Galerkin formulation for \eqref{eq:nonconservative-system} reads as follows \cite{RHEBERGEN2008,cheng2022bound,ZHANG2025}: find $\boldsymbol{u}_h$ with evolved components in $V_h^k$ and any stationary prescribed components supplied as data, such that, for any test function $\boldsymbol{\phi} \in (V_h^k)^m$,
 \begin{align}\label{eq:standard-pcdg-form}
&\frac{d}{dt}\int_{I_j} \boldsymbol{u}_h \cdot \boldsymbol{\phi} \, dx
=
\int_{I_j} \boldsymbol{f}(\boldsymbol{u}_h)\cdot \boldsymbol{\phi}_x \, dx
- \widehat{\boldsymbol{f}}_{j+\frac{1}{2}}\cdot \boldsymbol{\phi}^-_{j+\frac{1}{2}}
+ \widehat{\boldsymbol{f}}_{j-\frac{1}{2}}\cdot \boldsymbol{\phi}^+_{j-\frac{1}{2}} \\
&\quad
- \int_{I_j} G(\boldsymbol{u}_h)(\boldsymbol{u}_h)_x \cdot \boldsymbol{\phi} \, dx
+ \int_{I_j} \boldsymbol{s}(\boldsymbol{u}_h,\boldsymbol{\sigma})\cdot \boldsymbol{\phi} \, dx
- \frac{1}{2}\boldsymbol{\phi}^-_{j+\frac{1}{2}}\cdot \boldsymbol{g}_{\boldsymbol{\Phi},j+\frac{1}{2}}
- \frac{1}{2}\boldsymbol{\phi}^+_{j-\frac{1}{2}}\cdot \boldsymbol{g}_{\boldsymbol{\Phi},j-\frac{1}{2}}.\nonumber
\end{align}
Here $\widehat{\boldsymbol{f}}_{j+\frac{1}{2}}=\widehat{\boldsymbol{f}}(\boldsymbol{u}^-_{j+\frac{1}{2}},\boldsymbol{u}^+_{j+\frac{1}{2}})$ is a consistent numerical flux for the conservative part, $\boldsymbol{u}^\pm_{j+\frac{1}{2}}$ are the left and right traces of the numerical solution at the interface $x_{j+\frac{1}{2}}$, and the test function $\boldsymbol{\phi}$ is taken from the same approximation space. This is the semi-discrete PCDG weak form with the source term retained explicitly. When prescribed data form stationary rows, their numerical flux and fluctuation rows are zero and only the evolved components are time advanced.

The interface contribution of the non-conservative product is defined through the path integral
\begin{equation}
\boldsymbol{g}_{\boldsymbol{\Phi},j+\frac{1}{2}}
:=
\int_0^1
G(\boldsymbol{\Phi}_{j+\frac{1}{2}}(s))
(\boldsymbol{\Phi}_{j+\frac{1}{2}})_s(s)\,ds,
\label{eq:path-integral-term}
\end{equation}
where
\begin{equation}
\boldsymbol{\Phi}_{j+\frac{1}{2}}(s)=\boldsymbol{\Phi}(s;\boldsymbol{u}^-_{j+\frac{1}{2}},\boldsymbol{u}^+_{j+\frac{1}{2}}).
\label{eq:interface-path}
\end{equation}
The path constructions used for the two applications are specified in Section~\ref{se:applications}.

\begin{remark}
The semi-discrete formulation \eqref{eq:standard-pcdg-form} is consistent and path-conservative in the sense of Par\'es \cite{pares2006numerical}. Indeed, it can be written in fluctuation form with fluctuations that vanish for coincident states and whose sum equals the DLM integral of the quasilinear operator along the prescribed path. The detailed fluctuation representation and verification are given in Appendix~\ref{app:standard-pcdg-path-conservation}.
\end{remark}

\section{The well-balanced PCDG-GHR scheme}
\label{se:ghr}
This section constructs a high-order path-conservative discontinuous Galerkin (PCDG) scheme that exactly preserves the prescribed equilibrium states. The interface treatment uses generalized hydrostatic reconstruction (GHR) \cite{CASTRO2007}, whereas the interior treatment extends the reference-equilibrium discretization of Xing \cite{XING2014536}. Together, these two components eliminate equilibrium residuals at cell interfaces and within cells.

For the selected target equilibrium, the numerical state is decomposed as
\begin{equation}
\boldsymbol{u} = \boldsymbol{u}^e + \boldsymbol{u}^r,
\label{eq:equilibrium-decomposition}
\end{equation}
where $\boldsymbol{u}^e$ denotes the equilibrium component associated with the target equilibrium state and $\boldsymbol{u}^r$ is the corresponding perturbation component. In the semi-discrete setting, this leads to a cellwise reference equilibrium state $\boldsymbol{u}^e$, polynomial in its evolved components, and its perturbation part $\boldsymbol{u}^r = \boldsymbol{u}_h - \boldsymbol{u}^e$. The construction of $\boldsymbol{u}^e$ will be specified in the subsequent subsections.

The nonlinear fluxes, non-conservative products, and sources in the general construction below are evaluated at the numerical state $\boldsymbol{u}_h$ and the reference equilibrium state $\boldsymbol{u}^e$. The cellwise recovery of this state is specified below.

A semi-discrete well-balanced PCDG scheme is defined as follows: find $\boldsymbol{u}_h$ with the same evolved and prescribed components as above such that, for any test function $\boldsymbol{\phi} \in (V_h^k)^m$,
\begin{equation}
\begin{aligned}
\frac{d}{dt}\int_{I_j} \boldsymbol{u}_h \cdot \boldsymbol{\phi} \, dx
=\quad
&\int_{I_j} \boldsymbol{f}(\boldsymbol{u}_h)\cdot \boldsymbol{\phi}_x \, dx
- \widehat{\boldsymbol{f}}_{j+\frac{1}{2}}^{\,\mathrm{eff},L}\cdot \boldsymbol{\phi}_{j+\frac{1}{2}}^-
+ \widehat{\boldsymbol{f}}_{j-\frac{1}{2}}^{\,\mathrm{eff},R}\cdot \boldsymbol{\phi}_{j-\frac{1}{2}}^+ \\
&- \int_{I_j} G(\boldsymbol{u}_h)(\boldsymbol{u}_h)_x \cdot \boldsymbol{\phi} \, dx
- \frac{1}{2}\boldsymbol{\phi}_{j+\frac{1}{2}}^- \cdot \boldsymbol{g}_{\boldsymbol{\Phi},j+\frac{1}{2}}^{\,*}
- \frac{1}{2}\boldsymbol{\phi}_{j-\frac{1}{2}}^+ \cdot \boldsymbol{g}_{\boldsymbol{\Phi},j-\frac{1}{2}}^{\,*} \\
&+ \mathcal{R}_{j}^{\mathrm{wb}}(\boldsymbol{u}_h,\boldsymbol{u}^e,\boldsymbol{\phi}),
\end{aligned}
\label{eq:well-balanced-residual-form}
\end{equation}
where $\widehat{\boldsymbol{f}}^{\,\mathrm{eff},L}$ and $\widehat{\boldsymbol{f}}^{\,\mathrm{eff},R}$ are the effective interface fluxes produced by the GHR procedure, $\boldsymbol{g}_{\boldsymbol{\Phi}}^{\,*}$ is the modified path integral evaluated from the reconstructed interface states, and the well-balanced residual $\mathcal{R}_{j}^{\mathrm{wb}}(\boldsymbol{u}_h,\boldsymbol{u}^e,\boldsymbol{\phi})$ is defined, in accordance with \cite{XING2014536}, by
\begin{equation}
\begin{aligned}
\mathcal{R}_{j}^{\mathrm{wb}}(\boldsymbol{u}_h,\boldsymbol{u}^e,\boldsymbol{\phi})
=\quad
&-\int_{I_j} \boldsymbol{f}(\boldsymbol{u}^e)\cdot \boldsymbol{\phi}_x \, dx
+ \boldsymbol{f}(\boldsymbol{u}_{j+\frac{1}{2}}^{e,-})\cdot \boldsymbol{\phi}_{j+\frac{1}{2}}^-
- \boldsymbol{f}(\boldsymbol{u}_{j-\frac{1}{2}}^{e,+})\cdot \boldsymbol{\phi}_{j-\frac{1}{2}}^+ \\
&+ \int_{I_j} G(\boldsymbol{u}^e)(\boldsymbol{u}^e)_x \cdot \boldsymbol{\phi} \, dx
+ \int_{I_j}\bigl[\boldsymbol{s}(\boldsymbol{u}_h,\boldsymbol{\sigma})-\boldsymbol{s}(\boldsymbol{u}^e,\boldsymbol{\sigma})\bigr]\cdot\boldsymbol{\phi}\,dx.
\end{aligned}
\label{eq:well-balanced-residual-definition}
\end{equation}

By collecting the equilibrium contributions into $\mathcal{R}_{j}^{\mathrm{wb}}$ and then rearranging terms, one obtains the following equivalent perturbation form:
\begin{align}
&\frac{d}{dt}\int_{I_j} \boldsymbol{u}_h \cdot \boldsymbol{\phi} \, dx
\quad\nonumber\\
&=\int_{I_j} \bigl[\boldsymbol{f}(\boldsymbol{u}_h)-\boldsymbol{f}(\boldsymbol{u}^e)\bigr]\cdot \boldsymbol{\phi}_x \, dx 
- \bigl[\widehat{\boldsymbol{f}}_{j+\frac{1}{2}}^{\,\mathrm{eff},L} - \boldsymbol{f}(\boldsymbol{u}_{j+\frac{1}{2}}^{e,-})\bigr]\cdot \boldsymbol{\phi}_{j+\frac{1}{2}}^- 
+ \bigl[\widehat{\boldsymbol{f}}_{j-\frac{1}{2}}^{\,\mathrm{eff},R} - \boldsymbol{f}(\boldsymbol{u}_{j-\frac{1}{2}}^{e,+})\bigr]\cdot \boldsymbol{\phi}_{j-\frac{1}{2}}^+ \nonumber\\
&- \int_{I_j} \bigl[G(\boldsymbol{u}_h)(\boldsymbol{u}_h)_x - G(\boldsymbol{u}^e)(\boldsymbol{u}^e)_x\bigr]\cdot \boldsymbol{\phi} \, dx
+ \int_{I_j}\bigl[\boldsymbol{s}(\boldsymbol{u}_h,\boldsymbol{\sigma})-\boldsymbol{s}(\boldsymbol{u}^e,\boldsymbol{\sigma})\bigr]\cdot\boldsymbol{\phi}\,dx \nonumber\\
&- \frac{1}{2}\boldsymbol{\phi}_{j+\frac{1}{2}}^- \cdot \boldsymbol{g}_{\boldsymbol{\Phi},j+\frac{1}{2}}^{\,*}
- \frac{1}{2}\boldsymbol{\phi}_{j-\frac{1}{2}}^+ \cdot \boldsymbol{g}_{\boldsymbol{\Phi},j-\frac{1}{2}}^{\,*}.
\label{eq:equivalent-perturbation-form}
\end{align}
Equations \eqref{eq:well-balanced-residual-form} and \eqref{eq:equivalent-perturbation-form} are therefore two equivalent formulations of the same semi-discrete well-balanced PCDG method.

The following subsections specify the GHR interface terms and the reference-equilibrium interior residual in \eqref{eq:well-balanced-residual-form}--\eqref{eq:equivalent-perturbation-form}, followed by the fully discrete time update.

\subsection{Interface treatment by generalized hydrostatic reconstruction}
The interface treatment is based on the GHR technique of \cite{CASTRO2007}. Interface states are reconstructed from the equilibrium variables associated with the target steady state. When the discrete left and right equilibrium variables agree, the reconstruction places both states at the same point on a common reference datum. The interface fluctuations $\boldsymbol{D}^+$ and $\boldsymbol{D}^-$ in \eqref{eq:abstract-fluctuation-form} then vanish by consistency, suppressing spurious waves at equilibrium.

In the semi-discrete PCDG framework, GHR defines the effective numerical fluxes $\widehat{\boldsymbol{f}}^{\,\mathrm{eff}}$ and modified path integrals $\boldsymbol{g}_{\boldsymbol{\Phi},j+\frac{1}{2}}^{\,*}$ so that the interface contribution vanishes at the prescribed equilibrium. Section~\ref{subse:state} defines the reconstructed states, and Section~\ref{subse:flux} constructs the corresponding effective fluxes and path integrals.

\subsubsection{State reconstruction at interfaces}
\label{subse:state}
We use the equilibrium-variable transformation introduced in \eqref{eq:equilibrium-variable-transform}. Thus, for the target equilibrium family, the equilibrium variables $\boldsymbol{v}$ remain constant along the selected equilibrium state, while the inverse map $\tau^{-1}$ recovers states on the selected admissible branch. The interface reconstruction below applies this conversion after placing the left and right traces on a common geometric or material datum. The model-specific volume evaluation rules are specified separately in Section~\ref{se:applications}.

At $x_{j+\frac{1}{2}}$, let $\boldsymbol{u}_{h,j+\frac{1}{2}}^\pm$, $\boldsymbol{u}_{j+\frac{1}{2}}^{e,\pm}$, $\boldsymbol{u}_{j+\frac{1}{2}}^{\mathrm{eq},\pm}$, and $\boldsymbol{u}_{j+\frac{1}{2}}^{r,\pm}$ denote the left and right traces of the DG solution, the reference equilibrium projection, the recovered equilibrium state, and the perturbation $\boldsymbol{u}^r=\boldsymbol{u}_h-\boldsymbol{u}^e$, respectively. The equilibrium recovery and decomposition are detailed below. With $\boldsymbol{\sigma}_{j+\frac{1}{2}}^\pm=\boldsymbol{\sigma}(x_{j+\frac{1}{2}}^\pm)$ denoting the one-sided prescribed data, the interface equilibrium variables are defined by
\begin{equation}
\boldsymbol{v}_{j+\frac{1}{2}}^\pm
=
\tau\!\left(\boldsymbol{u}_{j+\frac{1}{2}}^{\mathrm{eq},\pm}+\boldsymbol{u}_{j+\frac{1}{2}}^{r,\pm},\boldsymbol{\sigma}_{j+\frac{1}{2}}^\pm\right).
\label{eq:ghr-interface-equilibrium-traces}
\end{equation}
To compare the two traces under a common geometric or material configuration, we choose a common reference datum $\boldsymbol{\sigma}_{j+\frac{1}{2}}^*$. This value is selected so that both inverse states are solvable on the prescribed admissible branches, namely
\begin{equation}
\tau^{-1}(\boldsymbol{v}_{j+\frac{1}{2}}^-,\boldsymbol{\sigma}_{j+\frac{1}{2}}^*)\in\Omega,
\qquad
\tau^{-1}(\boldsymbol{v}_{j+\frac{1}{2}}^+,\boldsymbol{\sigma}_{j+\frac{1}{2}}^*)\in\Omega.
\label{eq:ghr-reference-datum-v2}
\end{equation}

The generalized hydrostatic reconstruction is then performed by keeping the equilibrium variable fixed and recovering the states on the common reference datum $\boldsymbol{\sigma}_{j+\frac{1}{2}}^*$:
\begin{equation}
\boldsymbol{u}_{j+\frac{1}{2}}^{*,\pm}
=
\tau^{-1}(\boldsymbol{v}_{j+\frac{1}{2}}^\pm,\boldsymbol{\sigma}_{j+\frac{1}{2}}^*).
\label{eq:ghr-reconstructed-conservative-state-v2}
\end{equation}
The corresponding augmented states $\boldsymbol{w}_{j+\frac{1}{2}}^{*,\pm}$ are obtained from these reconstructed states and the common datum by the augmentation rule in Section~\ref{subsec:general-form}, without duplicating prescribed components already included in the state.
By construction, the reconstructed states $\boldsymbol{u}_{j+\frac{1}{2}}^{*,\pm}$ preserve the equilibrium-variable values of the corresponding interface traces used in \eqref{eq:ghr-interface-equilibrium-traces}, but they are now evaluated with identical prescribed data. These reconstructed states are therefore the natural inputs for the interface numerical fluxes and the non-conservative path integrals.

\subsubsection{Evaluation of numerical fluxes and path integrals}
\label{subse:flux}
Once the reconstructed states $\boldsymbol{u}_{j+\frac{1}{2}}^{*,-}$ and $\boldsymbol{u}_{j+\frac{1}{2}}^{*,+}$ are obtained, we evaluate the conservative part of the interface contribution by the Lax--Friedrichs flux
\begin{equation}
\widehat{\boldsymbol{f}}_{LF}(\boldsymbol{u}_{j+\frac{1}{2}}^{*,-},\boldsymbol{u}_{j+\frac{1}{2}}^{*,+})
=
\frac{1}{2}\bigl(\boldsymbol{f}(\boldsymbol{u}_{j+\frac{1}{2}}^{*,-})+\boldsymbol{f}(\boldsymbol{u}_{j+\frac{1}{2}}^{*,+})\bigr)
-\frac{\alpha}{2}\bigl(\boldsymbol{u}_{j+\frac{1}{2}}^{*,+}-\boldsymbol{u}_{j+\frac{1}{2}}^{*,-}\bigr),
\label{eq:ghr-lf-flux-v2}
\end{equation}
and
\begin{equation}
    \alpha=\max_{\boldsymbol{u}}\{|\lambda_1(\boldsymbol{u})|,\dots,|\lambda_{m}(\boldsymbol{u})|\},
\end{equation}
with $\lambda_i(\boldsymbol{u})$ the eigenvalues of the Jacobian matrix $\mathcal{A}$ in \eqref{eq:quasilinear-matrix}, and the maximum is taken globally over the computational domain.
To match the fluctuation form induced by \eqref{eq:well-balanced-residual-form}, we define the one-sided effective numerical fluxes by
\begin{align}
\widehat{\boldsymbol{f}}_{j+\frac{1}{2}}^{\,\mathrm{eff},L}
&=
\widehat{\boldsymbol{f}}_{LF}(\boldsymbol{u}_{j+\frac{1}{2}}^{*,-},\boldsymbol{u}_{j+\frac{1}{2}}^{*,+})
-\boldsymbol{f}(\boldsymbol{u}_{j+\frac{1}{2}}^{*,-})
+\boldsymbol{f}(\boldsymbol{u}_{h,j+\frac{1}{2}}^-),
\label{eq:ghr-effective-flux-left-v2}
\\
\widehat{\boldsymbol{f}}_{j+\frac{1}{2}}^{\,\mathrm{eff},R}
&=
\widehat{\boldsymbol{f}}_{LF}(\boldsymbol{u}_{j+\frac{1}{2}}^{*,-},\boldsymbol{u}_{j+\frac{1}{2}}^{*,+})
-\boldsymbol{f}(\boldsymbol{u}_{j+\frac{1}{2}}^{*,+})
+\boldsymbol{f}(\boldsymbol{u}_{h,j+\frac{1}{2}}^+).
\label{eq:ghr-effective-flux-right-v2}
\end{align}

\begin{remark}
The original hydrostatic reconstruction of Audusse et al.~\cite{Audusse2004AFA} is written in terms of two one-sided interface fluxes, since the interface correction is distributed to the left and right sides of each cell interface; see Eq.~(2.16) of \cite{Audusse2004AFA}. A commonly used alternative is to introduce a single modified flux
\begin{equation}
\widehat{\boldsymbol{f}}_{j+\frac{1}{2}}^{\,\mathrm{mod}}
=
\frac{1}{2}\bigl(\boldsymbol{f}(\boldsymbol{u}_{h,j+\frac{1}{2}}^-)+\boldsymbol{f}(\boldsymbol{u}_{h,j+\frac{1}{2}}^+)\bigr)
-\frac{\alpha}{2}\bigl(\boldsymbol{u}_{j+\frac{1}{2}}^{*,+}-\boldsymbol{u}_{j+\frac{1}{2}}^{*,-}\bigr).
\label{eq:remark31-modified-flux}
\end{equation}
This modified flux is related to the one-sided effective fluxes by
\begin{equation}
\widehat{\boldsymbol{f}}_{j+\frac{1}{2}}^{\,\mathrm{mod}}
=
\frac{1}{2}\Bigl(
\widehat{\boldsymbol{f}}_{j+\frac{1}{2}}^{\,\mathrm{eff},L}
+
\widehat{\boldsymbol{f}}_{j+\frac{1}{2}}^{\,\mathrm{eff},R}
\Bigr).
\label{eq:remark31-average-relation}
\end{equation}
Therefore, \eqref{eq:remark31-modified-flux} is not identical to either one-sided effective flux taken separately; rather, it is their symmetric average. In this sense, the modified flux form is an algebraic recombination of the original hydrostatic reconstruction formulation of Audusse et al.
\end{remark}

At the common datum, we select a reconstructed-state path from the path family in Section~\ref{subsec:dlm-path} that connects the two reconstructed states:
\begin{equation}
\boldsymbol{\Phi}_{j+\frac{1}{2}}^{*}(s)
=
\boldsymbol{\Phi}\bigl(s;\boldsymbol{u}_{j+\frac{1}{2}}^{*,-},
\boldsymbol{u}_{j+\frac{1}{2}}^{*,+}\bigr),
\qquad
s\in[0,1].
\label{eq:ghr-reconstructed-path-v2}
\end{equation}
If prescribed data are components of $\boldsymbol{u}$, they remain fixed at $\boldsymbol{\sigma}^*$ throughout this path.
The corresponding modified path integral is
\begin{equation}
\boldsymbol{g}_{\boldsymbol{\Phi},j+\frac{1}{2}}^{\,*}
=
\int_{0}^{1}
G(\boldsymbol{\Phi}_{j+\frac{1}{2}}^{*}(s))(\boldsymbol{\Phi}_{j+\frac{1}{2}}^{*})_{s}(s)\,ds.
\label{eq:ghr-modified-path-integral-v2}
\end{equation}
The integral in \eqref{eq:ghr-modified-path-integral-v2} is evaluated with a quadrature rule appropriate to the selected path. The complete model-specific connections and quadrature choices are specified in Section~\ref{se:applications}. Section~\ref{subse:wbin} shows that \eqref{eq:ghr-effective-flux-left-v2}--\eqref{eq:ghr-modified-path-integral-v2} cancel the interface contribution at the target equilibrium.

\subsubsection{Well-balanced property at the interface}
\label{subse:wbin}
We next establish exact interface balance for the generalized hydrostatic reconstruction of Section~\ref{subse:state}. Assume that, at equilibrium, the left and right interface traces satisfy
\begin{equation}
\boldsymbol{v}_{j+\frac{1}{2}}^-
=
\boldsymbol{v}_{j+\frac{1}{2}}^+.
\label{eq:ghr-equilibrium-variable-coincidence}
\end{equation}
This discrete interface condition must be verified for the selected equilibrium representation; it does not follow solely from constancy of the equilibrium variables in the continuous equilibrium state. Here the selected admissible branches are compatible if, when the left and right equilibrium variables agree, inversion at the same common datum returns the same reconstructed state. Since both reconstructed states are recovered on the same common reference datum $\boldsymbol{\sigma}_{j+\frac{1}{2}}^*$ and such compatible admissible branches, it follows from \eqref{eq:ghr-reconstructed-conservative-state-v2} that
\begin{equation}
\boldsymbol{u}_{j+\frac{1}{2}}^{*,-}
=
\boldsymbol{u}_{j+\frac{1}{2}}^{*,+}
=:\boldsymbol{u}_{j+\frac{1}{2}}^*.
\label{eq:ghr-star-state-coincidence}
\end{equation}

Consequently, the dissipative part of the local Lax--Friedrichs flux vanishes, and \eqref{eq:ghr-lf-flux-v2} reduces to
\begin{equation}
\widehat{\boldsymbol{f}}_{LF}(\boldsymbol{u}_{j+\frac{1}{2}}^{*,-},\boldsymbol{u}_{j+\frac{1}{2}}^{*,+})
=
\boldsymbol{f}(\boldsymbol{u}_{j+\frac{1}{2}}^*).
\label{eq:ghr-lf-flux-equilibrium}
\end{equation}
Substituting \eqref{eq:ghr-star-state-coincidence} and \eqref{eq:ghr-lf-flux-equilibrium} into \eqref{eq:ghr-effective-flux-left-v2}--\eqref{eq:ghr-effective-flux-right-v2}, we obtain
\begin{equation}
\widehat{\boldsymbol{f}}_{j+\frac{1}{2}}^{\,\mathrm{eff},L}
=
\boldsymbol{f}(\boldsymbol{u}_{h,j+\frac{1}{2}}^-),
\qquad
\widehat{\boldsymbol{f}}_{j+\frac{1}{2}}^{\,\mathrm{eff},R}
=
\boldsymbol{f}(\boldsymbol{u}_{h,j+\frac{1}{2}}^+).
\label{eq:ghr-effective-flux-equilibrium}
\end{equation}
Thus, at the target equilibrium, the one-sided effective fluxes reduce to the physical fluxes evaluated at the corresponding input traces.

For the non-conservative contribution, \eqref{eq:ghr-star-state-coincidence} and the consistency condition \eqref{eq:path-consistency} imply that the reconstructed-state path \eqref{eq:ghr-reconstructed-path-v2} is constant and therefore has zero derivative. Thus, \eqref{eq:ghr-modified-path-integral-v2} gives
\begin{equation}
\boldsymbol{g}_{\boldsymbol{\Phi},j+\frac{1}{2}}^{\,*}
=
0.
\label{eq:ghr-path-integral-equilibrium}
\end{equation}
Therefore, both the conservative and non-conservative interface corrections vanish in the sense required by the target equilibrium state.

This conclusion can also be interpreted in the fluctuation framework of \eqref{eq:abstract-fluctuation-form}. Indeed, the interface contribution in \eqref{eq:well-balanced-residual-form} corresponds to the modified fluctuations
\begin{equation}
\boldsymbol{D}_{j+\frac{1}{2}}^-
=
\widehat{\boldsymbol{f}}_{j+\frac{1}{2}}^{\,\mathrm{eff},L}
-\boldsymbol{f}(\boldsymbol{u}_{h,j+\frac{1}{2}}^-)
+\frac{1}{2}\boldsymbol{g}_{\boldsymbol{\Phi},j+\frac{1}{2}}^{\,*},
\qquad
\boldsymbol{D}_{j+\frac{1}{2}}^+
=
\boldsymbol{f}(\boldsymbol{u}_{h,j+\frac{1}{2}}^+)
-\widehat{\boldsymbol{f}}_{j+\frac{1}{2}}^{\,\mathrm{eff},R}
+\frac{1}{2}\boldsymbol{g}_{\boldsymbol{\Phi},j+\frac{1}{2}}^{\,*}.
\label{eq:ghr-modified-fluctuations}
\end{equation}
Combining \eqref{eq:ghr-effective-flux-equilibrium} and \eqref{eq:ghr-path-integral-equilibrium} yields
\begin{equation}
\boldsymbol{D}_{j+\frac{1}{2}}^-
=
0,
\qquad
\boldsymbol{D}_{j+\frac{1}{2}}^+
=
0.
\label{eq:ghr-vanishing-fluctuations}
\end{equation}
Therefore, the GHR interface fluctuations vanish at the target equilibrium, and the scheme introduces no spurious waves at cell interfaces.

\subsection{Interior treatment: Volume integrals via reference equilibrium decomposition}

GHR eliminates the interface imbalance, but polynomial projection can still prevent pointwise balance among the flux derivative $\boldsymbol{f}(\boldsymbol{u})_x$, the non-conservative product $G(\boldsymbol{u})\boldsymbol{u}_x$, and the source $\boldsymbol{s}(\boldsymbol{u},\boldsymbol{\sigma})$ within a cell. We therefore extend Xing's reference-equilibrium strategy \cite{XING2014536} to the present path-conservative setting. Decomposing the numerical solution into equilibrium and perturbation components removes the projection-induced equilibrium residual from the volume terms.

\subsubsection{Recovery of cellwise reference equilibrium values}
\label{subse:recovery}
Because projected equilibrium data generally fail to satisfy the equilibrium relations pointwise, the interior treatment first recovers a cellwise equilibrium representation. Let $\boldsymbol{u}_j(x):=\boldsymbol{u}_h|_{I_j}$ and let $\boldsymbol{\sigma}_j(x):=\boldsymbol{\sigma}(x)|_{I_j}$ denote the prescribed datum used in spatial evaluations on $I_j$; the latter is a projected representation only when the discretization projects the data. In cellwise formulas without a subscript, $\boldsymbol{\sigma}$ denotes the same datum. We adopt a reference-equilibrium recovery strategy in the spirit of Noelle, Xing, and Shu \cite{NOELLE200729}.

Let $\bar{\boldsymbol{v}}_j$ denote the constant equilibrium variable to be recovered in cell $I_j$. The components of $\bar{\boldsymbol{v}}_j$ are determined by matching a model-dependent set of cell averages. We denote the corresponding matching map by $\boldsymbol{\mu}(\boldsymbol{u},\boldsymbol{\sigma})$; its components are chosen to identify the selected equilibrium branch and may depend on the prescribed data. From the projected data in cell $I_j$, we compute
\begin{equation}
\bar{\boldsymbol{\mu}}_j
=
\frac{1}{\Delta x_j}\int_{I_j}\boldsymbol{\mu}(\boldsymbol{u}_j(x),\boldsymbol{\sigma}_j(x))\,dx.
\label{eq:wb-interior-matching-average}
\end{equation}
The cellwise reference equilibrium variable $\bar{\boldsymbol{v}}_j$ is then defined by the matching condition
\begin{equation}
\bar{\boldsymbol{\mu}}_j
=
\frac{1}{\Delta x_j}\int_{I_j}
\boldsymbol{\mu}\bigl(\tau^{-1}(\bar{\boldsymbol{v}}_j,\boldsymbol{\sigma}_j(x)),\boldsymbol{\sigma}_j(x)\bigr)\,dx,
\label{eq:wb-interior-reference-v}
\end{equation}
where the inverse is taken on the selected admissible branch. This condition states that the recovered equilibrium state reproduces the chosen cell averages of the projected numerical state.

By the cell-average matching condition and the equilibrium relations, the recovered $\bar{\boldsymbol{v}}_j$ identifies the cellwise reference equilibrium state.

\subsubsection{Reference equilibrium decomposition and residual approximation}

The recovery procedure of Section~\ref{subse:recovery} yields the reference equilibrium variable $\bar{\boldsymbol{v}}_j$ in each cell. We denote the corresponding exact equilibrium state by
\begin{equation}
\boldsymbol{u}^{\mathrm{eq}}_j(x)
:=
\tau^{-1}(\bar{\boldsymbol{v}}_j,\boldsymbol{\sigma}_j(x)).
\label{eq:wb-interior-exact-equilibrium}
\end{equation}
Since the evolved components of this state are generally not polynomials of degree $k$, we define their reference equilibrium projection by
\begin{equation}
\boldsymbol{u}^e|_{I_j}
=
\Pi_j\boldsymbol{u}^{\mathrm{eq}}_j(x)
\quad\text{in the evolved components},
\label{eq:wb-interior-reference-projection}
\end{equation}
where $\Pi_j$ denotes the cellwise $L^2$ projection onto $V_h^k$. Any stationary prescribed components are retained unchanged in both $\boldsymbol{u}_h$ and $\boldsymbol{u}^e$.

Since the evolved components of $\boldsymbol{u}_h$ and $\boldsymbol{u}^e$ belong to the same polynomial space and their prescribed components agree, the numerical solution can be decomposed as
\begin{equation}
\boldsymbol{u}_h|_{I_j}
=
\boldsymbol{u}^e|_{I_j}
+
\boldsymbol{u}^r|_{I_j},
\qquad
\boldsymbol{u}^r|_{I_j}
=
\boldsymbol{u}_h|_{I_j}
-\boldsymbol{u}^e|_{I_j}.
\label{eq:wb-interior-decomposition}
\end{equation}
This is the reference equilibrium decomposition used throughout the interior treatment. Here $\boldsymbol{u}^{\mathrm{eq}}_j$ is the exact recovered equilibrium state, $\boldsymbol{u}^e$ is its DG projection in the evolved components with any prescribed components retained, and $\boldsymbol{u}^r$ is polynomial in the evolved components and zero in the prescribed components.

When an explicit source term is present, it is decomposed in the abstract form
\begin{equation}
\int_{I_j} \boldsymbol{s}(\boldsymbol{u}_h,\boldsymbol{\sigma})\cdot\boldsymbol{\phi}\,dx
=
\int_{I_j} \boldsymbol{s}(\boldsymbol{u}^e,\boldsymbol{\sigma})\cdot\boldsymbol{\phi}\,dx
+
\int_{I_j}\bigl[\boldsymbol{s}(\boldsymbol{u}_h,\boldsymbol{\sigma})-\boldsymbol{s}(\boldsymbol{u}^e,\boldsymbol{\sigma})\bigr]\cdot\boldsymbol{\phi}\,dx.
\label{eq:wb-interior-source-decomposition}
\end{equation}
This decomposition separates the equilibrium contribution from the nonequilibrium correction without relying on any model-specific linear structure of the source term.

The exact equilibrium state $\boldsymbol{u}^{\mathrm{eq}}_j$ satisfies the weak balance relation
\begin{align*}
\int_{I_j} \boldsymbol{s}\bigl(\boldsymbol{u}^{\mathrm{eq}}_j,\boldsymbol{\sigma}\bigr)\cdot\boldsymbol{\phi}\,dx
&=-\int_{I_j}\boldsymbol{f}\bigl(\boldsymbol{u}^{\mathrm{eq}}_j\bigr)\cdot\boldsymbol{\phi}_x\,dx
+\boldsymbol{f}\bigl(\boldsymbol{u}^{\mathrm{eq},-}_{j+\frac{1}{2}}\bigr)\cdot\boldsymbol{\phi}_{j+\frac{1}{2}}^- \\
&\quad-\boldsymbol{f}\bigl(\boldsymbol{u}^{\mathrm{eq},+}_{j-\frac{1}{2}}\bigr)\cdot\boldsymbol{\phi}_{j-\frac{1}{2}}^+
+\int_{I_j}G\bigl(\boldsymbol{u}^{\mathrm{eq}}_j\bigr)\bigl(\boldsymbol{u}^{\mathrm{eq}}_j\bigr)_x\cdot\boldsymbol{\phi}\,dx.
\end{align*}
Applying the $L^2$ projection argument of \cite{XING2014536} to the preceding relation, including its non-conservative term, gives
\begin{align}
\int_{I_j} \boldsymbol{s}(\boldsymbol{u}^e,\boldsymbol{\sigma})\cdot\boldsymbol{\phi}\,dx
&=-\int_{I_j}\boldsymbol{f}(\boldsymbol{u}^e)\cdot\boldsymbol{\phi}_x\,dx
+\boldsymbol{f}(\boldsymbol{u}_{j+\frac{1}{2}}^{e,-})\cdot\boldsymbol{\phi}_{j+\frac{1}{2}}^- \nonumber\\
&\quad-\boldsymbol{f}(\boldsymbol{u}_{j-\frac{1}{2}}^{e,+})\cdot\boldsymbol{\phi}_{j-\frac{1}{2}}^+
+\int_{I_j}G(\boldsymbol{u}^e)(\boldsymbol{u}^e)_x\cdot\boldsymbol{\phi}\,dx
+O(\Delta x_j^{k+1}).
\label{eq:wb-interior-projected-equilibrium-relation}
\end{align}

Here, \(O(\Delta x_j^{k+1})\) denotes the cellwise weak projection error for uniformly bounded DG test functions, assuming \(\boldsymbol{u}^{\mathrm{eq}}_j\in W^{k+1,\infty}(I_j)\), sufficiently smooth prescribed data, and accurate quadrature; it does not by itself imply a global convergence order.

Combining \eqref{eq:wb-interior-source-decomposition} and \eqref{eq:wb-interior-projected-equilibrium-relation}, we approximate the source term integral by
\begin{equation}
\begin{aligned}
\int_{I_j} \boldsymbol{s}(\boldsymbol{u}_h,\boldsymbol{\sigma})\cdot\boldsymbol{\phi}\,dx
\approx
&-\int_{I_j} \boldsymbol{f}(\boldsymbol{u}^e)\cdot\boldsymbol{\phi}_x\,dx + \boldsymbol{f}(\boldsymbol{u}_{j+\frac{1}{2}}^{e,-})\cdot\boldsymbol{\phi}_{j+\frac{1}{2}}^-
- \boldsymbol{f}(\boldsymbol{u}_{j-\frac{1}{2}}^{e,+})\cdot\boldsymbol{\phi}_{j-\frac{1}{2}}^+ \\
&+ \int_{I_j} G(\boldsymbol{u}^e)(\boldsymbol{u}^e)_x\cdot\boldsymbol{\phi}\,dx
+ \int_{I_j}\bigl[\boldsymbol{s}(\boldsymbol{u}_h,\boldsymbol{\sigma})-\boldsymbol{s}(\boldsymbol{u}^e,\boldsymbol{\sigma})\bigr]\cdot\boldsymbol{\phi}\,dx,
\end{aligned}
\label{eq:wb-interior-source-approximation}
\end{equation}
The right-hand side is the term used in \eqref{eq:well-balanced-residual-definition}.

\subsubsection{Well-balanced evaluation of the interior residual}

Using the approximation \eqref{eq:wb-interior-source-approximation} in \eqref{eq:well-balanced-residual-form} leads to the perturbation formulation \eqref{eq:equivalent-perturbation-form}. Its cell-interior part is given by
\begin{align}
\mathcal{I}_{j}^{\mathrm{int}}(\boldsymbol{u}_h,\boldsymbol{u}^e,\boldsymbol{\phi})
=
&\int_{I_j} \bigl[\boldsymbol{f}(\boldsymbol{u}_h)-\boldsymbol{f}(\boldsymbol{u}^e)\bigr]\cdot \boldsymbol{\phi}_x\,dx \nonumber\\
&- \int_{I_j} \bigl[G(\boldsymbol{u}_h)(\boldsymbol{u}_h)_x-G(\boldsymbol{u}^e)(\boldsymbol{u}^e)_x\bigr]\cdot \boldsymbol{\phi}\,dx\nonumber\\
&+ \int_{I_j}\bigl[\boldsymbol{s}(\boldsymbol{u}_h,\boldsymbol{\sigma})-\boldsymbol{s}(\boldsymbol{u}^e,\boldsymbol{\sigma})\bigr]\cdot\boldsymbol{\phi}\,dx.
\label{eq:wb-interior-perturbation-volume}
\end{align}
Here $\mathcal{I}_{j}^{\mathrm{int}}(\boldsymbol{u}_h,\boldsymbol{u}^e,\boldsymbol{\phi})$ collects the cell-interior terms evaluated relative to the reference equilibrium state. Their evolved components differ by the polynomial perturbation $\boldsymbol{u}^r=\boldsymbol{u}_h-\boldsymbol{u}^e$; their prescribed components agree.

Hence, if the numerical state lies on the target equilibrium so that $\boldsymbol{u}_h=\boldsymbol{u}^e$ in $I_j$, then
\begin{equation}
\mathcal{I}_{j}^{\mathrm{int}}(\boldsymbol{u}_h,\boldsymbol{u}^e,\boldsymbol{\phi})
=
0
\qquad
\text{for all }\boldsymbol{\phi}\in (V_h^k)^m.
\label{eq:wb-interior-vanishing}
\end{equation}
Therefore, the volume contribution introduces no spurious cell-interior imbalance at the prescribed equilibrium state, which is the volume counterpart of the interface cancellation established in Section~\ref{subse:wbin}.

\subsection{Summary of the fully discrete scheme}
\label{sub:full}
For hyperbolic conservation laws and non-conservative hyperbolic systems, high-order accuracy in both space and time is commonly obtained by coupling a high-order spatial discretization, such as a DG or finite volume method, with strong stability preserving Runge--Kutta (SSP-RK) or multistep time discretizations \cite{Gottlieb2001,ShuChiWang1988}. In the present work, we use the third-order SSP-RK method to advance the semi-discrete well-balanced PCDG scheme in time.

For notational convenience, we denote the semi-discrete well-balanced spatial operator by $\mathcal{L}_{j}^{\mathrm{wb}}$. More precisely, for any test function $\boldsymbol{\phi}\in (V_h^k)^m$, we write
\begin{equation*}
\begin{aligned}
\frac{d}{dt}\int_{I_j} \boldsymbol{u}_h\cdot\boldsymbol{\phi}\,dx
&=
\mathcal{L}_{j}^{\mathrm{wb}}(\boldsymbol{u}_h,\boldsymbol{\phi}),
\\
\mathcal{L}_{j}^{\mathrm{wb}}(\boldsymbol{u}_h,\boldsymbol{\phi})
&=
\int_{I_j} \boldsymbol{f}(\boldsymbol{u}_h)\cdot \boldsymbol{\phi}_x\,dx
-\widehat{\boldsymbol{f}}_{j+\frac{1}{2}}^{\,\mathrm{eff},L}\cdot\boldsymbol{\phi}_{j+\frac{1}{2}}^-
+\widehat{\boldsymbol{f}}_{j-\frac{1}{2}}^{\,\mathrm{eff},R}\cdot\boldsymbol{\phi}_{j-\frac{1}{2}}^+
\\
&\quad
-\int_{I_j}G(\boldsymbol{u}_h)(\boldsymbol{u}_h)_x\cdot\boldsymbol{\phi}\,dx
-\frac{1}{2}\boldsymbol{\phi}_{j+\frac{1}{2}}^-\cdot \boldsymbol{g}_{\boldsymbol{\Phi},j+\frac{1}{2}}^{\,*}
-\frac{1}{2}\boldsymbol{\phi}_{j-\frac{1}{2}}^+\cdot \boldsymbol{g}_{\boldsymbol{\Phi},j-\frac{1}{2}}^{\,*}
\\
&\quad
+\mathcal{R}_{j}^{\mathrm{wb}}(\boldsymbol{u}_h,\boldsymbol{u}^e,\boldsymbol{\phi}).
\end{aligned}
\end{equation*}
The equivalent perturbation formulation \eqref{eq:equivalent-perturbation-form} gives the same operator after collecting the equilibrium contributions into $\mathcal{R}_{j}^{\mathrm{wb}}$. In each evaluation of $\mathcal{L}_{j}^{\mathrm{wb}}$, the interface contribution is computed by the GHR effective fluxes and the modified path integral, while the cell-interior contribution is represented by $\mathcal{I}_{j}^{\mathrm{int}}(\boldsymbol{u}_h,\boldsymbol{u}^e,\boldsymbol{\phi})$ in \eqref{eq:wb-interior-perturbation-volume}.

Given $\boldsymbol{u}_h^n$, the fully discrete SSP-RK3 scheme is defined as follows:
\begin{equation}
\begin{aligned}
\int_{I_j} \boldsymbol{u}_h^{(1)}\cdot\boldsymbol{\phi}\,dx
&=
\int_{I_j} \boldsymbol{u}_h^n\cdot\boldsymbol{\phi}\,dx
+\Delta t\,\mathcal{L}_{j}^{\mathrm{wb}}(\boldsymbol{u}_h^n,\boldsymbol{\phi}),
\\
\int_{I_j} \boldsymbol{u}_h^{(2)}\cdot\boldsymbol{\phi}\,dx
&=
\frac{3}{4}\int_{I_j} \boldsymbol{u}_h^n\cdot\boldsymbol{\phi}\,dx
+\frac{1}{4}\left(
\int_{I_j} \boldsymbol{u}_h^{(1)}\cdot\boldsymbol{\phi}\,dx
+\Delta t\,\mathcal{L}_{j}^{\mathrm{wb}}(\boldsymbol{u}_h^{(1)},\boldsymbol{\phi})
\right),
\\
\int_{I_j} \boldsymbol{u}_h^{n+1}\cdot\boldsymbol{\phi}\,dx
&=
\frac{1}{3}\int_{I_j} \boldsymbol{u}_h^n\cdot\boldsymbol{\phi}\,dx
+\frac{2}{3}\left(
\int_{I_j} \boldsymbol{u}_h^{(2)}\cdot\boldsymbol{\phi}\,dx
+\Delta t\,\mathcal{L}_{j}^{\mathrm{wb}}(\boldsymbol{u}_h^{(2)},\boldsymbol{\phi})
\right).
\end{aligned}
\label{eq:ssp-rk3-ghr-pcdg}
\end{equation}
At every Runge--Kutta stage, the GHR reconstruction, effective numerical fluxes, modified path integrals, and well-balanced residual $\mathcal{R}_{j}^{\mathrm{wb}}$ are recomputed from the current stage value. The time update remains explicit. Its evaluation may require independent local algebraic equilibrium recovery, but it does not require a coupled implicit solve for the DG coefficients.

For solutions containing discontinuities or sharp gradients, we use a total variation bounded (TVB) limiter \cite{COCKBURN1989,COCKBURN1998}. The TVB modified-minmod test, controlled by the parameter $M$, identifies the cells in which limiting is applied after each SSP-RK stage; prescribed components remain fixed. Exact equilibrium preservation is retained only when the limiter is inactive on the discrete equilibrium. Accordingly, we either choose $M$ so that the equilibrium is not flagged or disable the limiter in equilibrium-preservation tests.

\subsection{Compatibility with positivity-preserving limiters}
\label{subsec:positivity-preserving-limiters}

For a physical variable to be treated by a positivity-preserving scaling limiter, its evolution equation must have a conservative cell-average update: the corresponding rows of the non-conservative matrix and source must contribute no cell-average term, and the well-balanced interface and interior treatments must preserve the associated mass-flux structure. Under these compatibility conditions, a scaling limiter may be applied to the physical variable after each SSP-RK stage while preserving its cell average \cite{xing2010positivity,zhang2012maximum}. A fully discrete positivity result additionally requires a positive cell-average update under suitable conditions on the evaluation and reconstructed states, numerical flux, positivity nodes, and time step; such a result is not established here for a general non-conservative system.

\section{Path conservation and well-balanced property}
\label{sec:exactly-well-balanced-property}

We establish path conservation and exact well-balancing of the PCDG-GHR method under the stated conditions. We first analyze the GHR interface fluctuations along complete DLM paths and then prove equilibrium preservation at the semi-discrete and SSP-RK fully discrete levels.

\subsection{Path conservation of the GHR interface fluctuations}
\label{subsec:ghr-path-conservation}

Remark~2.1 verifies path conservation for the standard PCDG interface fluctuations associated with the homogeneous jump operator. For the GHR treatment, the reconstructed-state path \eqref{eq:ghr-reconstructed-path-v2} is completed by the two equilibrium connections associated with reconstruction, following the idea of \cite{CASTRO2007}. Below we suppress the interface index and write $\boldsymbol{w}^{\pm}$ for the augmented forms of the interface traces used in the equilibrium-variable transformation \eqref{eq:ghr-interface-equilibrium-traces}, using the rule in Section~\ref{subsec:general-form}.

\begin{theorem}
\label{thm:ghr-path-conservation}
Consider the admissible GHR reconstruction of Section~\ref{subse:state}. Let $\widetilde{\boldsymbol{\Phi}}$ be the complete path obtained by traversing three segments in order: $\boldsymbol{P}^-$ connects $\boldsymbol{w}^-$ to $\boldsymbol{w}^{*,-}$ with $\boldsymbol{v}^-$ fixed; $\boldsymbol{P}^{0}$ connects the two reconstructed states at the common datum; and $\boldsymbol{P}^+$ connects $\boldsymbol{w}^{*,+}$ to $\boldsymbol{w}^+$ with $\boldsymbol{v}^+$ fixed. Each segment is Lipschitz and parametrized on $[0,1]$. The middle segment is the augmented form of \eqref{eq:ghr-reconstructed-path-v2}, with the prescribed data fixed at $\boldsymbol{\sigma}^*$.

Assume that the outer connections exist along admissible stationary equilibrium curves of the homogeneous augmented interface system in Section~\ref{subsec:general-form}, and that the complete path family satisfies DLM path consistency and the local Lipschitz regularity conditions of Definition~1 of \cite{CASTRO2007}. Then the modified fluctuations \eqref{eq:ghr-modified-fluctuations} are path conservative with respect to $\widetilde{\boldsymbol{\Phi}}$ and vanish for coincident complete traces.
\end{theorem}

\begin{proof}
As in Remark~2.1, substituting the effective fluxes \eqref{eq:ghr-effective-flux-left-v2}--\eqref{eq:ghr-effective-flux-right-v2} into \eqref{eq:ghr-modified-fluctuations} and adding the fluctuations cancels the original physical flux terms and the common numerical flux. By \eqref{eq:ghr-modified-path-integral-v2} and \eqref{eq:quasilinear-matrix}, the chain rule gives
\begin{equation}
\boldsymbol{D}^-+\boldsymbol{D}^+
=\boldsymbol{f}(\boldsymbol{u}^{*,+})-\boldsymbol{f}(\boldsymbol{u}^{*,-})
+\boldsymbol{g}_{\boldsymbol{\Phi}}^{\,*}
=\int_0^1\mathcal{A}\bigl(\boldsymbol{\Phi}^{*}(s)\bigr)
\frac{d\boldsymbol{\Phi}^{*}}{ds}(s)\,ds.
\label{eq:ghr-modified-fluctuation-sum}
\end{equation}
The prescribed data are fixed along $\boldsymbol{P}^{0}$, so their derivatives contribute nothing to the augmented matrix product. Thus the preceding integral, with any additional zero block for external prescribed components, is the integral of $\widetilde{\mathcal{A}}$ along $\boldsymbol{P}^{0}$.

Since the outer segments follow stationary equilibrium curves of the augmented system, the steady balance relation gives
\[
\widetilde{\mathcal{A}}\bigl(\boldsymbol{P}^{\pm}(s)\bigr)
\frac{d\boldsymbol{P}^{\pm}}{ds}(s)=0
\quad\text{for almost every }s\in[0,1].
\]
Consequently, both outer path integrals vanish. By additivity under successive traversal of the three segments, the complete path integral is the sum of their individual integrals and therefore equals the middle integral. Hence
\begin{equation}
\begin{pmatrix}\boldsymbol{D}^-+\boldsymbol{D}^+\\\boldsymbol{0}\end{pmatrix}
=\int_0^1\widetilde{\mathcal{A}}\bigl(\boldsymbol{P}^{0}(s)\bigr)\frac{d\boldsymbol{P}^{0}}{ds}(s)\,ds
=\int_0^1\widetilde{\mathcal{A}}\bigl(\widetilde{\boldsymbol{\Phi}}(s)\bigr)\frac{d\widetilde{\boldsymbol{\Phi}}}{ds}(s)\,ds.
\label{eq:ghr-complete-path-conservation}
\end{equation}
This proves path conservation along the selected complete DLM path. For coincident complete traces, the assumed DLM consistency makes the complete path, and hence its middle segment, constant, so the two reconstructed states coincide. The Lax--Friedrichs flux \eqref{eq:ghr-lf-flux-v2} then equals their common physical flux, and \eqref{eq:ghr-modified-path-integral-v2} vanishes. Substitution into the effective fluxes and \eqref{eq:ghr-modified-fluctuations} gives $\boldsymbol{D}^- = \boldsymbol{D}^+ = \boldsymbol{0}$.
\end{proof}

\subsection{Well-balanced property}
\label{subsec:ghr-exactly-well-balanced}

The statements below apply with periodic boundary conditions, or with nonperiodic boundary treatments whose exterior traces represent the same target equilibrium and whose numerical boundary terms satisfy the interface cancellation identities \eqref{eq:ghr-effective-flux-equilibrium} and \eqref{eq:ghr-path-integral-equilibrium}.

\begin{theorem}
\label{thm:semi-discrete-exactly-well-balanced}
Assume that the numerical solution is initialized with the DG representation of a target equilibrium in the sense of \eqref{eq:steady-balance}. Suppose that the reference equilibrium recovery in each cell exactly returns this discrete equilibrium state, so that
\begin{equation}
\boldsymbol{u}_h|_{I_j}=\boldsymbol{u}^e|_{I_j},
\qquad
\text{for all cells } I_j.
\label{eq:equilibrium-polynomial-coincidence}
\end{equation}
Assume also that the discrete interface traces satisfy \eqref{eq:ghr-equilibrium-variable-coincidence} and that the common-datum inverse states are defined on compatible admissible branches, meaning that coincident equilibrium variables inverted at the same common datum return the same reconstructed state.
Then the semi-discrete well-balanced PCDG scheme \eqref{eq:well-balanced-residual-form} satisfies
\begin{equation}
\mathcal{L}_{j}^{\mathrm{wb}}(\boldsymbol{u}_h,\boldsymbol{\phi})=0,
\qquad
\forall \boldsymbol{\phi}\in (V_h^k)^m,
\qquad
\text{for all cells } I_j.
\label{eq:semi-discrete-wb-operator-vanishes}
\end{equation}
Consequently, the semi-discrete right-hand side vanishes exactly at the equilibrium state.
\end{theorem}

\begin{proof}
By the assumed discrete interface matching condition and the GHR interface analysis in Section~\ref{subse:wbin}, the reconstructed states coincide on the common reference datum. Consequently,
\begin{equation}
\widehat{\boldsymbol{f}}_{j+\frac{1}{2}}^{\,\mathrm{eff},L}
=
\boldsymbol{f}(\boldsymbol{u}_{h,j+\frac{1}{2}}^-),
\qquad
\widehat{\boldsymbol{f}}_{j+\frac{1}{2}}^{\,\mathrm{eff},R}
=
\boldsymbol{f}(\boldsymbol{u}_{h,j+\frac{1}{2}}^+),
\qquad
\boldsymbol{g}_{\boldsymbol{\Phi},j+\frac{1}{2}}^{\,*}=0.
\label{eq:equilibrium-interface-cancellation}
\end{equation}
These identities are precisely \eqref{eq:ghr-effective-flux-equilibrium} and \eqref{eq:ghr-path-integral-equilibrium}.

Substituting \eqref{eq:equilibrium-interface-cancellation} into \eqref{eq:well-balanced-residual-form}, the semi-discrete operator in cell \(I_j\) becomes
\begin{equation}
\begin{aligned}
\mathcal{L}_{j}^{\mathrm{wb}}(\boldsymbol{u}_h,\boldsymbol{\phi})
=\quad
&\int_{I_j} \boldsymbol{f}(\boldsymbol{u}_h)\cdot \boldsymbol{\phi}_x\,dx
- \boldsymbol{f}(\boldsymbol{u}_{h,j+\frac{1}{2}}^-)\cdot\boldsymbol{\phi}_{j+\frac{1}{2}}^-
+ \boldsymbol{f}(\boldsymbol{u}_{h,j-\frac{1}{2}}^+)\cdot\boldsymbol{\phi}_{j-\frac{1}{2}}^+
\\
&-\int_{I_j}G(\boldsymbol{u}_h)(\boldsymbol{u}_h)_x\cdot\boldsymbol{\phi}\,dx
+\mathcal{R}_{j}^{\mathrm{wb}}(\boldsymbol{u}_h,\boldsymbol{u}^e,\boldsymbol{\phi}).
\end{aligned}
\label{eq:operator-after-interface-cancellation}
\end{equation}
By \eqref{eq:equilibrium-polynomial-coincidence}, \(\boldsymbol{u}_h=\boldsymbol{u}^e\) in \(I_j\), and the nonequilibrium source correction satisfies
\begin{equation}
\int_{I_j}\bigl[\boldsymbol{s}(\boldsymbol{u}_h,\boldsymbol{\sigma})-\boldsymbol{s}(\boldsymbol{u}^e,\boldsymbol{\sigma})\bigr]\cdot\boldsymbol{\phi}\,dx=0.
\label{eq:equilibrium-source-difference-vanishes}
\end{equation}
Using \eqref{eq:well-balanced-residual-definition} and \(\boldsymbol{u}_h=\boldsymbol{u}^e\), we obtain
\begin{equation}
\begin{aligned}
\mathcal{R}_{j}^{\mathrm{wb}}(\boldsymbol{u}_h,\boldsymbol{u}^e,\boldsymbol{\phi})
=\quad
&-\int_{I_j} \boldsymbol{f}(\boldsymbol{u}_h)\cdot \boldsymbol{\phi}_x\,dx
+ \boldsymbol{f}(\boldsymbol{u}_{h,j+\frac{1}{2}}^-)\cdot\boldsymbol{\phi}_{j+\frac{1}{2}}^-
- \boldsymbol{f}(\boldsymbol{u}_{h,j-\frac{1}{2}}^+)\cdot\boldsymbol{\phi}_{j-\frac{1}{2}}^+
\\
&+\int_{I_j}G(\boldsymbol{u}_h)(\boldsymbol{u}_h)_x\cdot\boldsymbol{\phi}\,dx.
\end{aligned}
\label{eq:equilibrium-residual-cancellation}
\end{equation}
Combining \eqref{eq:operator-after-interface-cancellation} and \eqref{eq:equilibrium-residual-cancellation}, all conservative flux, non-conservative product, and boundary-flux contributions cancel exactly. Thus \eqref{eq:semi-discrete-wb-operator-vanishes} holds for every test function \(\boldsymbol{\phi}\) and every cell \(I_j\).
\end{proof}

\begin{theorem}
\label{thm:fully-discrete-exactly-well-balanced}
Under the assumptions of Theorem~\ref{thm:semi-discrete-exactly-well-balanced}, with the limiter turned off, the fully discrete well-balanced PCDG scheme \eqref{eq:ssp-rk3-ghr-pcdg} exactly preserves the equilibrium state. In particular, if \(\boldsymbol{u}_h^n\) is such an equilibrium state, then
\begin{equation}
\boldsymbol{u}_h^{n+1}=\boldsymbol{u}_h^n.
\label{eq:fully-discrete-equilibrium-preservation}
\end{equation}
\end{theorem}

\begin{proof}
The result follows by induction over the SSP-RK stages. By Theorem~\ref{thm:semi-discrete-exactly-well-balanced}, the spatial operator \(\mathcal{L}_{j}^{\mathrm{wb}}\) vanishes whenever a stage value equals the equilibrium state. Each stage in \eqref{eq:ssp-rk3-ghr-pcdg} is therefore a convex combination of the same discrete equilibrium and remains unchanged.
\end{proof}

\section{Applications}
\label{se:applications}

This section specializes the general PCDG-GHR construction to two models:  the two-layer shallow water and blood flow equations. In each case, the equilibrium variables, local recovery, interface datum, path, and nonlinear evaluation states are chosen on an admissible branch.

\subsection{Two-layer shallow water equations}
\label{subsec:application-two-layer}

For the one-dimensional two-layer shallow water equations, we consider
\begin{equation}
\begin{cases}
(h_1)_t+(m_1)_x=0, \\[1mm]
(m_1)_t+\left(\dfrac{m_1^2}{h_1}+\dfrac{1}{2}gh_1^2\right)_x
=-gh_1(h_2)_x-gh_1b_x, \\[2mm]
(h_2)_t+(m_2)_x=0, \\[1mm]
(m_2)_t+\left(\dfrac{m_2^2}{h_2}+\dfrac{1}{2}gh_2^2\right)_x
=-grh_2(h_1)_x-gh_2b_x,
\end{cases}
\label{eq:two-layer-system}
\end{equation}
where $h_i$ and $m_i=h_i u_i$ denote the depth and discharge of the $i$th layer, respectively, $g$ is the gravitational acceleration, $r$ is the density ratio, and $b(x)$ is the prescribed bottom topography. In the notation of \eqref{eq:nonconservative-system}, $\boldsymbol{u}=(h_1,m_1,h_2,m_2)^T$, $\boldsymbol{\sigma}=b$, and
\begin{equation}
\boldsymbol{f}(\boldsymbol{u})
=
\begin{pmatrix}
m_1\\[1mm]
\dfrac{m_1^2}{h_1}+\dfrac{1}{2}gh_1^2\\[2mm]
m_2\\[1mm]
\dfrac{m_2^2}{h_2}+\dfrac{1}{2}gh_2^2
\end{pmatrix},
\quad
G(\boldsymbol{u})
=
\begin{pmatrix}
0&0&0&0\\
0&0&gh_1&0\\
0&0&0&0\\
grh_2&0&0&0
\end{pmatrix},
\quad
\boldsymbol{s}(\boldsymbol{u},b)
=
\begin{pmatrix}
0\\[1mm]
-gh_1b_x\\[1mm]
0\\[1mm]
-gh_2b_x
\end{pmatrix}.
\label{eq:two-layer-flux-source}
\end{equation}
Here the bottom is external to $\boldsymbol{u}$. On smooth cells, the source has the form $-C(\boldsymbol{u},b)b_x$ with $C(\boldsymbol{u},b)=(0,gh_1,0,gh_2)^T$. At a bottom jump, the augmented state vector is $\boldsymbol{w}=(\boldsymbol{u}^T,b)^T$, with augmented matrix $\widetilde{\mathcal{A}}=\left(\begin{smallmatrix}\mathcal{A}&C\\0&0\end{smallmatrix}\right)$; the bottom derivative is therefore represented by the augmented interface path rather than by a cellwise source evaluation at the jump.

\noindent\textbf{a. Equilibrium variables and admissible branch.}

With
\begin{equation}
H_2=h_2+b,
\qquad
H_1=h_1+H_2=h_1+h_2+b,
\label{eq:two-layer-surfaces}
\end{equation}
denoting the internal interface and the free surface, respectively, the moving-water equilibrium is characterized by
\begin{align*}
&m_1=\mathrm{const},
\qquad E_1=\frac{m_1^2}{2h_1^2}+g(h_1+H_2)=\mathrm{const},\\
&m_2=\mathrm{const},
\qquad
E_2=\frac{m_2^2}{2h_2^2}+g(rh_1+H_2)=\mathrm{const},
\end{align*}
and the still-water equilibrium is the special case
\begin{equation}
m_1=m_2=0,
\qquad
h_1+H_2=\mathrm{const},
\qquad
rh_1+H_2=\mathrm{const}.
\label{eq:still-water-equilibrium}
\end{equation}

Here $\boldsymbol{\sigma}=b$ and $\boldsymbol{v}=(m_1,E_1,m_2,E_2)^T$ in \eqref{eq:equilibrium-variable-transform}. The inverse equilibrium map recovers the layer depths for specified $(m_1,E_1,m_2,E_2)$ and local bottom value $b$ on a selected branch. That branch must give positive layer depths, maintain the ordering of the free surface and internal interface, and remain in the prescribed hyperbolic flow regime.

\noindent\textbf{b. Local reference equilibrium recovery.}

Writing $\boldsymbol{u}_h=(h_{1,h},m_{1,h},h_{2,h},m_{2,h})^T$, let $\boldsymbol{\mu}(\boldsymbol{u},b)=(m_1,m_2,h_1,h_2+b)^T$ and $\bar{\boldsymbol{v}}_j=(\bar m_{1,j},\bar E_{1,j},\bar m_{2,j},\bar E_{2,j})^T$. The two discharge components are fixed by
\begin{equation}
\bar m_{1,j}=\frac{1}{\Delta x_j}\int_{I_j}m_{1,h}(x)\,dx,
\qquad
\bar m_{2,j}=\frac{1}{\Delta x_j}\int_{I_j}m_{2,h}(x)\,dx.
\end{equation}
The two energy components are then determined by the nonlinear matching system
\begin{equation}
\begin{cases}
\displaystyle
\frac{1}{\Delta x_j}\int_{I_j}h_1^{\mathrm{eq}}(x;\bar{\boldsymbol{v}}_j)\,dx
=\frac{1}{\Delta x_j}\int_{I_j}h_{1,h}(x)\,dx,\\[3mm]
\displaystyle
\frac{1}{\Delta x_j}\int_{I_j}\bigl(h_2^{\mathrm{eq}}(x;\bar{\boldsymbol{v}}_j)+b_j(x)\bigr)\,dx
=\frac{1}{\Delta x_j}\int_{I_j}\bigl(h_{2,h}(x)+b_j(x)\bigr)\,dx.
\end{cases}
\label{eq:two-layer-local-equilibrium-recovery}
\end{equation}
At each $x\in I_j$, the depths in \eqref{eq:two-layer-local-equilibrium-recovery} are recovered on the selected admissible branch from
\begin{equation}
\begin{cases}
\displaystyle
\bar E_{1,j}=\frac{\bar m_{1,j}^2}{2(h_1^{\mathrm{eq}})^2}
+g\bigl(h_1^{\mathrm{eq}}+h_2^{\mathrm{eq}}+b_j(x)\bigr),\\[2mm]
\displaystyle
\bar E_{2,j}=\frac{\bar m_{2,j}^2}{2(h_2^{\mathrm{eq}})^2}
+g\bigl(rh_1^{\mathrm{eq}}+h_2^{\mathrm{eq}}+b_j(x)\bigr).
\end{cases}
\label{eq:two-layer-pointwise-equilibrium-recovery}
\end{equation}
Equations \eqref{eq:two-layer-local-equilibrium-recovery}--\eqref{eq:two-layer-pointwise-equilibrium-recovery} give the model-specific form of \eqref{eq:wb-interior-matching-average}--\eqref{eq:wb-interior-reference-v}. The resulting equilibrium state $\boldsymbol{u}^{\mathrm{eq}}_j=\tau^{-1}(\bar{\boldsymbol{v}}_j,b_j)$ is projected to define $\boldsymbol{u}^e$ as in \eqref{eq:wb-interior-reference-projection}.

\noindent\textbf{c. GHR interface reconstruction and path integration.}

On smooth portions of the bottom, the source term in \eqref{eq:two-layer-flux-source} contains $b_x$, while bottom jumps are included in the GHR interface construction. At $x_{j+\frac{1}{2}}$, we choose $b_{j+\frac{1}{2}}^*=\max(b_{j+\frac{1}{2}}^-,b_{j+\frac{1}{2}}^+)$ whenever both reconstructed layer depths remain admissible. The two states in \eqref{eq:ghr-reconstructed-conservative-state-v2} are recovered on this common bottom datum. Along admissible outer equilibrium paths with fixed $(m_1,E_1,m_2,E_2)$, the steady relations give $\widetilde{\mathcal{A}}(\boldsymbol{w})\boldsymbol{w}_s=0$, as required in Theorem~\ref{thm:ghr-path-conservation}. For the reconstructed-state path in \eqref{eq:ghr-reconstructed-path-v2}, we use the linear parametrization
\begin{equation}
\boldsymbol{\Phi}_{j+\frac{1}{2}}^{*}(s)
=\boldsymbol{u}_{j+\frac{1}{2}}^{*,-}
+s\bigl(\boldsymbol{u}_{j+\frac{1}{2}}^{*,+}-\boldsymbol{u}_{j+\frac{1}{2}}^{*,-}\bigr),
\qquad s\in[0,1],
\label{eq:two-layer-linear-reconstructed-state-path}
\end{equation}
provided this segment stays in the admissible state set. The one-sided effective flux corrections account for the change to the common datum.

Along this linear path, the two-layer shallow water matrix $G$ is affine, so the midpoint rule is exact:
\begin{equation}
\boldsymbol{g}_{\boldsymbol{\Phi},j+\frac{1}{2}}^{\,*}
=
G(\bar{\boldsymbol{u}}_{j+\frac{1}{2}}^{*})
\bigl(\boldsymbol{u}_{j+\frac{1}{2}}^{*,+}-\boldsymbol{u}_{j+\frac{1}{2}}^{*,-}\bigr),
\qquad
\bar{\boldsymbol{u}}_{j+\frac{1}{2}}^{*}
=
\frac{1}{2}\bigl(\boldsymbol{u}_{j+\frac{1}{2}}^{*,-}+\boldsymbol{u}_{j+\frac{1}{2}}^{*,+}\bigr).
\label{eq:ghr-evaluated-path-integral-v2}
\end{equation}

\noindent\textbf{d. Nonlinear evaluation and limiter compatibility.}

The volume terms use the polynomial pair $(\boldsymbol{u}_h,\boldsymbol{u}^e)$, while the interface equilibrium variables for GHR are evaluated from $\boldsymbol{u}^{\mathrm{eq},\pm}+\boldsymbol{u}^{r,\pm}$ as in \eqref{eq:ghr-interface-equilibrium-traces}. The physical-flux correction terms retain the original DG polynomial traces. Under exact recovery of the same target equilibrium in adjacent cells and compatible admissible branches, $\boldsymbol{u}^r=0$ and the interface equilibrium variables coincide, so both volume and interface contributions cancel.

The two-layer shallow water equations provide a representative example of this compatibility. For this model, $\boldsymbol{\sigma}=b$, the layer depths $h_1$ and $h_2$ satisfy conservative mass equations, the first and third rows of $G(\boldsymbol{u})$ vanish, and the first and third components of $\boldsymbol{s}(\boldsymbol{u},b)$ are zero. Moreover, the GHR reconstruction preserves the layer discharges, and the reference equilibrium residual contributes no additional cell-average source term to the depth equations. Thus a scaling limiter may be applied to each physical layer depth about its cell average at the chosen positivity nodes after each SSP-RK stage, while the bottom remains fixed. If a pre-balanced variable such as $H_2=h_2+b$ is evolved, the limiter is applied to $h_2=H_2-b$. For the standard DG scaling construction and a two-layer positive-depth flux treatment, see \cite{xing2010positivity,kurganov2009central}, respectively; these references do not establish a positivity theorem for the present GHR discretization.

\subsection{Blood flow equations}
\label{subsec:application-blood-flow}

Earlier GHR-based high-order finite volume schemes for blood flow in vessels with varying mechanical properties were developed by M\"uller et al. \cite{muller2013blood}. For blood flow through a thin-walled deformable elastic vessel with possibly discontinuous mechanical and geometric properties, we consider the model in \cite{pimentel2023high},
\begin{equation}
\begin{cases}
A_t+q_x=0,\\[1mm]
q_t+\left(\dfrac{q^2}{A}\right)_x+\dfrac{A}{\rho}p_x=0,
\end{cases}
\label{eq:blood-flow-system}
\end{equation}
with the pressure law
\begin{equation}
p(x,t)=K(x)\phi\left(\frac{A(x,t)}{A_0(x)}\right)+p_e(x),
\qquad
\phi(a)=a^{m_{\mathrm p}}-a^{n_{\mathrm p}}.
\label{eq:blood-flow-pressure-law}
\end{equation}
Here $A$ is the cross-sectional area of the vessel, $q=Au$ is the volumetric flow rate, $u$ is the cross-sectionally averaged blood velocity, and $\rho$ is the constant fluid density. The prescribed functions $A_0(x)$, $K(x)$, and $p_e(x)$ denote the unloaded cross-sectional area, the vessel-wall stiffness coefficient, and the external pressure, respectively. The pressure-law exponents satisfy $m_{\mathrm p}>0$ and $n_{\mathrm p}\in(-2,0]$; the commonly used choices are $(m_{\mathrm p},n_{\mathrm p})=(1/2,0)$ for arteries and $(m_{\mathrm p},n_{\mathrm p})=(10,-3/2)$ for collapsible vessels such as veins.

To exhibit the non-conservative structure of \eqref{eq:blood-flow-system}, we collect the physical variables and prescribed vessel data in the augmented vector
\begin{equation}
\boldsymbol{w}=(A,q,K,A_0,p_e)^T,
\qquad
a=\frac{A}{A_0},
\label{eq:blood-flow-augmented-state}
\end{equation}
and the functions, defined for $a>0$ by
\begin{equation}
\Phi(a)=\int_1^a\phi(z)\,dz,
\qquad
\widetilde\Phi(a)=a\phi(a)-\Phi(a).
\label{eq:blood-flow-Phi-definitions}
\end{equation}
Thus,
\begin{equation}
\Phi(1)=\widetilde\Phi(1)=0,
\qquad
\Phi'(a)=\phi(a),
\qquad
\widetilde\Phi'(a)=a\phi'(a),
\qquad
\Phi(a)+\widetilde\Phi(a)=a\phi(a).
\label{eq:blood-flow-Phi-identities}
\end{equation}
This normalization follows \cite{pimentel2023high}.
The corresponding flux and non-conservative matrix are
\begin{equation}
\begin{aligned}
\boldsymbol{f}(\boldsymbol{w})
&=
\begin{pmatrix}
q\\[1mm]
\dfrac{q^2}{A}+\dfrac{K A_0}{\rho}\widetilde\Phi(a)\\[2mm]
0\\
0\\
0
\end{pmatrix},
\qquad
G(\boldsymbol{w})
&=
\begin{pmatrix}
0&0&0&0&0\\
0&0&\dfrac{A_0}{\rho}\Phi(a)&-\dfrac{K}{\rho}\widetilde\Phi(a)&\dfrac{A}{\rho}\\[2mm]
0&0&0&0&0\\
0&0&0&0&0\\
0&0&0&0&0
\end{pmatrix}.
\end{aligned}
\label{eq:blood-flow-flux-source}
\end{equation}
For the frictionless model considered here, the source term is $\boldsymbol{s}(\boldsymbol{w},\boldsymbol{\sigma})=\boldsymbol{0}$. In the general formulation, $\boldsymbol{u}$ is the five-component vector $\boldsymbol{w}$ itself: its last three components form the prescribed stationary block $\boldsymbol{\sigma}=(K,A_0,p_e)^T$. Thus $\boldsymbol{w}$ is already the augmented state vector and its augmented matrix is $\widetilde{\mathcal{A}}=\mathcal{A}(\boldsymbol{w})$; no second copy of $\boldsymbol{\sigma}$ is introduced. The prescribed components retain their values throughout the time update and are supplied directly from the vessel data in spatial evaluations; only $(A,q)$ has a nontrivial update.

\noindent\textbf{a. Equilibrium variables and admissible branch.}

In the equilibrium-variable transformation \eqref{eq:equilibrium-variable-transform}, we take $\boldsymbol{v}=(q,\Gamma)^T$. The moving-flow equilibrium is characterized by
\begin{equation}
q=\mathrm{const},
\qquad
\Gamma
=\frac{\rho q^2}{2A^2}
+K\phi\left(\frac{A}{A_0}\right)+p_e
=\mathrm{const},
\label{eq:blood-flow-moving-equilibrium}
\end{equation}
For $A>0$, $A_0>0$, and $K>0$, the local wave speed and Shapiro number are
\begin{equation}
c=\sqrt{\frac{Ka}{\rho}\phi'(a)},
\qquad
\mathrm{Sh}=\frac{|u|}{c}.
\label{eq:blood-flow-shapiro-number}
\end{equation}
The admissible inverse branch is selected according to the local flow regime; in particular, the moving equilibria considered below lie on the subcritical branch $\mathrm{Sh}<1$.
The blood-at-rest equilibrium is the special case
\begin{equation}
q=0,
\qquad
K\phi\left(\frac{A}{A_0}\right)+p_e=\mathrm{const}.
\label{eq:blood-flow-rest-equilibrium}
\end{equation}

\noindent\textbf{b. Local reference equilibrium recovery.}

Writing the physical components of $\boldsymbol{w}_h$ as $(A_h,q_h)$, let $\boldsymbol{\mu}(\boldsymbol{w},\boldsymbol{\sigma})=(q,A)^T$ and $\bar{\boldsymbol{v}}_j=(\bar q_j,\bar\Gamma_j)^T$. The discharge is fixed by its cell average, while $\bar\Gamma_j$ is determined by the scalar nonlinear matching equation
\begin{equation}
\bar q_j=\frac{1}{\Delta x_j}\int_{I_j}q_h(x)\,dx,
\qquad
\frac{1}{\Delta x_j}\int_{I_j}A^{\mathrm{eq}}(x;\bar q_j,\bar\Gamma_j)\,dx
=\frac{1}{\Delta x_j}\int_{I_j}A_h(x)\,dx.
\label{eq:blood-flow-local-equilibrium-recovery}
\end{equation}
where, at each spatial point, $A^{\mathrm{eq}}>0$ is the solution on the selected Shapiro branch of
\begin{equation}
\bar\Gamma_j
=\frac{\rho\bar q_j^2}{2(A^{\mathrm{eq}})^2}
+K_j(x)\phi\left(\frac{A^{\mathrm{eq}}}{A_{0,j}(x)}\right)+p_{e,j}(x).
\label{eq:blood-flow-local-area-recovery}
\end{equation}
Equations \eqref{eq:blood-flow-local-equilibrium-recovery}--\eqref{eq:blood-flow-local-area-recovery} give the model-specific form of \eqref{eq:wb-interior-matching-average}--\eqref{eq:wb-interior-reference-v}. The recovered equilibrium state $\boldsymbol{w}^{\mathrm{eq}}_j=\tau^{-1}(\bar{\boldsymbol{v}}_j,\boldsymbol{\sigma}_j)$ is projected in its physical components to obtain $\boldsymbol{w}^e$ as in \eqref{eq:wb-interior-reference-projection}. Here $\boldsymbol{w}_h$ and $\boldsymbol{w}^e$ correspond to $\boldsymbol{u}_h$ and $\boldsymbol{u}^e$ in the general construction, while the prescribed block $(K,A_0,p_e)^T$ is retained unchanged in both states and in all spatial evaluations.

\noindent\textbf{c. GHR interface reconstruction and path integration.}

At an interface, let $\boldsymbol{\sigma}^\pm=(K^\pm,A_0^\pm,p_e^\pm)^T$ denote the exact one-sided prescribed data and let $(A^\pm,q^\pm)$ be the corresponding solution traces. Following \cite{pimentel2023high}, we choose the common intermediate datum
\begin{equation}
\boldsymbol{\sigma}^*=\begin{pmatrix}K^*\\ A_0^*\\ p_e^*\end{pmatrix},
\qquad
p_e^*=\min(p_e^-,p_e^+),
\qquad
A_0^*=\max(A_0^-,A_0^+),
\label{eq:blood-flow-star-parameters}
\end{equation}
with
\begin{equation}
K^*=
\begin{cases}
\max(K^-,K^+),
& A^-\le A_0^*,\ A^+<A_0^*,\\[1mm]
\min(K^-,K^+),
& A^-\ge A_0^*,\ A^+\ge A_0^*,\\[1mm]
\dfrac{K^-+K^+}{2},
& \text{otherwise}.
\end{cases}
\label{eq:blood-flow-K-star}
\end{equation}
With $\boldsymbol{v}^\pm=(q^\pm,\Gamma^\pm)^T$ and the common datum $\boldsymbol{\sigma}^*$, the reconstructed physical states are obtained from the general GHR relation \eqref{eq:ghr-reconstructed-conservative-state-v2} on the admissible branch associated with each original trace. Following the augmented construction in \cite{pimentel2023high}, the resulting connection consists of seven segments. Starting from the left state, three equilibrium segments change $p_e$, $A_0$, and $K$ successively to their common values while keeping $(q,\Gamma)$ fixed. Differentiating these equilibrium relations and using \eqref{eq:blood-flow-flux-source} gives $\mathcal{A}(\boldsymbol{w})\boldsymbol{w}_s=0$ on each admissible segment, as required for the outer paths of Theorem~\ref{thm:ghr-path-conservation}. A reconstructed-state segment at fixed $\boldsymbol{\sigma}^*$ connects the two reconstructed states, with the same linear formula as \eqref{eq:two-layer-linear-reconstructed-state-path} applied to the blood-flow states whenever the segment is admissible. Three equilibrium segments in reverse parameter order then connect the right reconstructed state to the original right state. The equilibrium segments have zero path integrals, whereas the modified path integral is evaluated on the reconstructed-state segment. Thus, discontinuities in the vessel data are treated in the full augmented state space already defined by $\boldsymbol{w}$, while \eqref{eq:ghr-reconstructed-path-v2} represents only the reconstructed-state segment of this model-specific path.

\noindent\textbf{d. Nonlinear evaluation and limiter compatibility.}

For nonlinear volume and interface evaluations, we use the equilibrium-lifted pair
\begin{equation}
\widetilde{\boldsymbol{w}}^{e}=\boldsymbol{w}^{\mathrm{eq}}_j,
\qquad
\widetilde{\boldsymbol{w}}_h
=\boldsymbol{w}^{\mathrm{eq}}_j+(\boldsymbol{w}_h-\boldsymbol{w}^e)
=\widetilde{\boldsymbol{w}}^{e}+\boldsymbol{w}^r.
\label{eq:blood-flow-lifted-evaluation-pair}
\end{equation}
Here $\boldsymbol{w}^{\mathrm{eq}}_j=\tau^{-1}(\bar{\boldsymbol{v}}_j,\boldsymbol{\sigma}_j(x))$ is the recovered augmented equilibrium state. The pair $(\widetilde{\boldsymbol{w}}_h,\widetilde{\boldsymbol{w}}^{e})$ replaces $(\boldsymbol{w}_h,\boldsymbol{w}^e)$ in all nonlinear spatial evaluations, including the volume terms and the GHR interface terms, while the time derivative, equilibrium recovery, and perturbation remain those of $\boldsymbol{w}_h$. The prescribed block $\boldsymbol{\sigma}=(K,A_0,p_e)^T$ is inserted unchanged, and the one-sided traces of $\widetilde{\boldsymbol{w}}_h$ are used in the GHR reconstruction and in the conditions defining $K^*$ in \eqref{eq:blood-flow-K-star}. Since $\widetilde{\boldsymbol{w}}_h-\widetilde{\boldsymbol{w}}^{e}=\boldsymbol{w}^r$, both evaluation states coincide with $\boldsymbol{w}^{\mathrm{eq}}_j$ when $\boldsymbol{w}^r=0$. Consequently, the interior and interface cancellations established for the general scheme remain valid under the same interface-matching and admissible-branch conditions.

The area equation $A_t+q_x=0$ has the corresponding row of the augmented non-conservative matrix equal to zero, and the GHR reconstruction preserves the discharge. This makes an area-scaling limiter compatible with the conservative cell-average update: it may be applied to the evolved polynomial $A_h$ about its cell average at the chosen positivity nodes after each SSP-RK stage. Because the nonlinear spatial operator uses the lifted area $\widetilde A_h$ from \eqref{eq:blood-flow-lifted-evaluation-pair}, admissibility of that evaluation state and of the GHR interface states must be checked separately. For the general scaling-limiter construction and a model-specific positivity-preserving blood-flow scheme, see \cite{zhang2012maximum,hernandezduenas2021blood}, respectively. This compatibility discussion does not establish positivity or hyperbolicity of the full augmented state.

\section{Numerical examples}
\label{se:nu}

We assess the PCDG-GHR method using the two one-dimensional models specified in Section~\ref{se:applications}: the two-layer shallow water equations and the blood flow equations. The numerical examples test high-order accuracy, exact equilibrium preservation, the resolution of small perturbations around equilibria, and the treatment of discontinuous flows.

Unless otherwise specified, we use piecewise quadratic $P^2$ polynomials in space and the third-order SSP-RK time discretization in \eqref{eq:ssp-rk3-ghr-pcdg}. The time step is chosen according to the CFL condition
\begin{equation}
\Delta t
:=
\frac{\mathrm{CFL}}{2k+1}\frac{\Delta x}{\max |\lambda_{\max}|},
\label{eq:dg-cfl-condition}
\end{equation}
where $k$ is the polynomial degree and $\lambda_{\max}$ denotes the largest characteristic speed over the computational domain. Unless otherwise specified, we take $\mathrm{CFL}=0.5$. The TVD limiter, equivalently the TVB limiter with $M=0$, is used for non-smooth tests when needed. In the accuracy tests and the exact equilibrium-preservation tests, the limiter is turned off or chosen to remain inactive so that it does not modify the discrete equilibrium state.

\subsection{Two-layer shallow water equations}

For the two-layer shallow water equations \eqref{eq:two-layer-system}, we test high-order accuracy, preservation of still-water and moving-water equilibria, small perturbations of these equilibria, interface propagation, and a Riemann problem over discontinuous bottom topography. Unless stated otherwise, the spatial and temporal discretizations follow the settings above; only model-specific parameters and initial data are given below.

\exampletitle{exmp:two-layer-accuracy}{Accuracy test}

To verify the high-order accuracy of the proposed well-balanced path-conservative DG scheme for a
smooth solution, we choose the same bottom topography and initial conditions as in \cite{kurganov2009central}, given by
\begin{equation}
    \begin{aligned}
    b(x) & =\sin^2(\pi x)-10,\quad h_1(x,0)=5+e^{\cos(2\pi x)}, \\
    H_2(x,0) & =-5-e^{\cos(2\pi x)},\quad m_1(x,0)=m_2(x,0)=0,
    \end{aligned}
\label{eq:two-layer-accuracy-initial-data}
\end{equation}
We impose 1-periodic boundary conditions on $[0,1]$. The gravitational constant is taken as $g=9.8$, and the density ratio is $r=0.98$. The simulations are run up to $t=0.01$, before shock formation. The numerical solution computed by the PCDG-GHR method on a refined mesh with $N_x=12800$ is used as the reference solution.

\begin{table}[!htb]
\centering
    \caption{Example~\ref{exmp:two-layer-accuracy}. $L^1$ errors and numerical orders of accuracy with initial conditions \eqref{eq:two-layer-accuracy-initial-data}, using $P^2$ polynomial.}
\label{tab:two-layer-accuracy-test}
\begin{tabular}{lllllllll}
\toprule
\multirow{2}{*}{$N_x$} & \multicolumn{2}{l}{$h_1$} & \multicolumn{2}{l}{$m_1$} & \multicolumn{2}{l}{$h_2$} & \multicolumn{2}{l}{$m_2$} \\ \cline{2-9} 
                      & $L_1$ error    & order    & $L_1$ error    & order    & $L_1$ error    & order    & $L_1$ error    & order    \\
    \midrule
   20  &  2.04e-04  &    /  &  7.27e-05  &   /  &  2.02e-04  &    /  &  7.36e-05  &   / \\
   40  &  2.53e-05  &    3.01  &  2.06e-05  &    1.82  &  2.51e-05  &    3.01  &  2.07e-05  &    1.83 \\
   80  &  3.30e-06  &    2.94  &  4.95e-06  &    2.06  &  3.28e-06  &    2.94  &  4.95e-06  &    2.06 \\
  160  &  4.54e-07  &    2.86  &  1.06e-06  &    2.22  &  4.51e-07  &    2.86  &  1.06e-06  &    2.22 \\
  320  &  5.88e-08  &    2.95  &  1.99e-07  &    2.42  &  5.84e-08  &    2.95  &  1.99e-07  &    2.42 \\
  640  &  6.05e-09  &    3.28  &  3.09e-08  &    2.69  &  6.01e-09  &    3.28  &  3.08e-08  &    2.69 \\
 1280  &  5.66e-10  &    3.42  &  4.09e-09  &    2.92  &  5.64e-10  &    3.41  &  4.09e-09  &    2.91 \\
 \bottomrule
\end{tabular}
\end{table}

The results in Table~\ref{tab:two-layer-accuracy-test} show the expected high-order accuracy for this smooth two-layer test. The errors of $h_1$ and $h_2$ exhibit nearly third-order convergence after the coarsest levels, and the convergence orders of $m_1$ and $m_2$ approach three under mesh refinement. This verifies the designed accuracy of the $P^2$ PCDG-GHR method for the two-layer shallow water equations.

\exampletitle{exmp:two-layer-still-equilibrium}{Still-water equilibrium preservation}
In this example, we verify the well-balanced property for the still-water equilibrium. The continuous bottom topography
\begin{equation}
    b(x)=\left\{
    \begin{array}
    {ll}0.25(\cos(10\pi(x-0.5))+1)-2, & \mathrm{if~}0.4<x<0.6, \\
    -2 & \text{otherwise,}
    \end{array}\right.
\label{eq:two-layer-still-continuous-bottom}
\end{equation}
and the discontinuous bottom topography
\begin{equation}
    b(x)=\left\{
    \begin{array}
    {ll}-1.5, & \mathrm{if~}x>0.5, \\
    -2, & \text{otherwise,}
    \end{array}\right.
\label{eq:two-layer-still-discontinuous-bottom}
\end{equation}
are considered. The initial condition is given by
\begin{equation}
    \begin{aligned}
    (h_1,m_1,H_2,m_2)(x,0)=(h_1,m_1,H_2,m_2)_{\mathrm{eq}}(x)=(1,0,-1,0),
    \end{aligned}
\label{eq:two-layer-still-equilibrium-initial-data}
\end{equation}
in the computational domain $[-0.2,1]$. In this set of two-layer equilibrium tests, we take $g=10.0$ and $r=0.98$. The numerical solutions are computed up to $t=0.1$ with $N_x=100$ uniform cells.

\begin{table}[!htb]
\centering
\caption{Example~\ref{exmp:two-layer-still-equilibrium}. $L_1$ and $L_\infty$ errors for the one-dimensional still-water equilibrium state with continuous bottom topography, $N_x=100$}
\label{tab:two-layer-still-continuous}
\begin{tabular}{lllllll}
\toprule
error      & $h_1$        & $m_1$        & $h_2$        & $m_2$        & $E_1$        & $E_2$        \\
\midrule
$L_1$      & 7.00e-15 & 2.18e-16 & 7.04e-15 & 2.07e-16 & 5.46e-16 & 2.00e-15 \\
$L_\infty$ & 4.16e-14 & 2.17e-15 & 4.20e-14 & 2.74e-15 & 9.99e-15 & 1.76e-14 \\
\bottomrule
\end{tabular}
\end{table}

\begin{table}[!htb]
\centering
\caption{Example~\ref{exmp:two-layer-still-equilibrium}. $L_1$ and $L_\infty$ errors for the one-dimensional still-water equilibrium state with discontinuous bottom topography, $N_x=100$}
\label{tab:two-layer-still-discontinuous}
\begin{tabular}{lllllll}
\toprule
error      & $h_1$        & $m_1$        & $h_2$        & $m_2$        & $E_1$        & $E_2$        \\
\midrule
$L_1$      & 5.46e-17 & 3.55e-18 & 6.25e-17 & 3.55e-18 & 6.34e-17 & 2.00e-16 \\
$L_\infty$ & 2.64e-16 & 2.96e-18 & 2.13e-16 & 2.96e-18 & 4.38e-16 & 1.67e-16 \\
\bottomrule
\end{tabular}
\end{table}

Tables~\ref{tab:two-layer-still-continuous} and~\ref{tab:two-layer-still-discontinuous} report the errors for the continuous and discontinuous bottom topographies, respectively. In both cases, the errors in the conservative and equilibrium variables remain at round-off level. This confirms that the PCDG-GHR method preserves the two-layer still-water equilibrium for both smooth and discontinuous topographies.

\exampletitle{exmp:two-layer-small-still}{Small perturbation of still-water equilibrium}

We next test whether the method can resolve a localized perturbation of the still-water equilibrium. The setup is the same as in the preceding still-water equilibrium test. Starting from the equilibrium state \eqref{eq:two-layer-still-equilibrium-initial-data}, we add a small perturbation to the upper-layer depth:

\begin{equation}
\begin{aligned}
    & h_1(x,0)=(h_1)_{\mathrm{eq}}(x)+\left\{
    \begin{array}
    {ll}0.00001, & \mathrm{if~}x\in[0.1,0.2], \\
    0, & \text{otherwise,}
    \end{array}\right. \\
    & H_2(x,0)=(H_2)_{\mathrm{eq}}(x),\quad m_1(x,0)=(m_1)_{\mathrm{eq}}(x),\quad m_2(x,0)=(m_2)_{\mathrm{eq}}(x).
\end{aligned}
\label{eq:two-layer-still-perturbation-initial-data}
\end{equation}
The stopping time is $t = 0.15$. We calculate the solution on a coarse mesh with $N_x=200$ cells and compare it with a reference solution computed on a refined mesh with $N_x=2000$ cells.

\begin{figure}[!htb]
\centering
\includegraphics[width=1\textwidth]{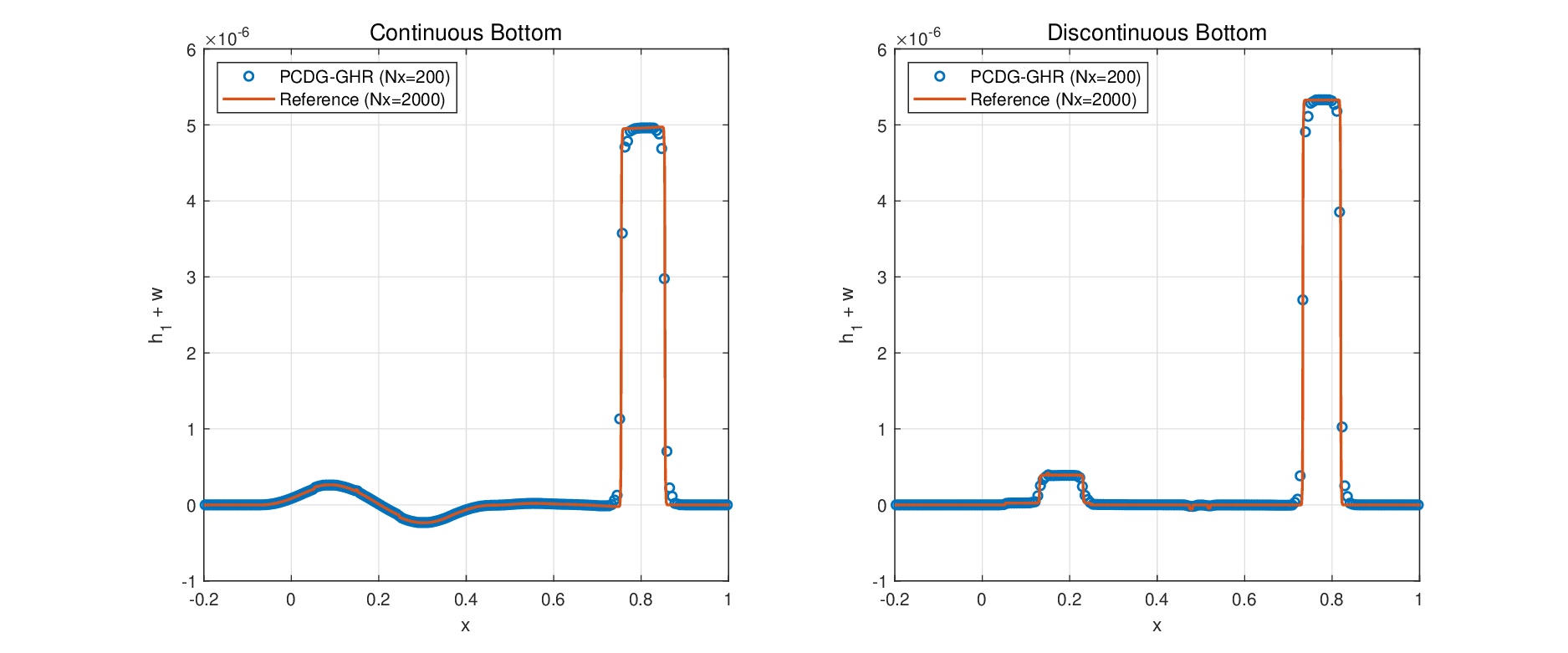}
\caption{Example~\ref{exmp:two-layer-small-still}. Small perturbation of the two-layer still-water equilibrium. The PCDG-GHR solution with $N_x=200$ is compared with the reference solution with $N_x=2000$.}
\label{fig:two-layer-small-still}
\end{figure}

Figure~\ref{fig:two-layer-small-still} shows the perturbation of the free surface $H_1$ for the continuous and discontinuous bottom topographies. The coarse-mesh solution agrees well with the refined reference solution and captures the propagation of the small disturbance from the underlying equilibrium state. No visible spurious oscillations are generated near the bottom discontinuity, which further illustrates the well-balanced resolution of small perturbations around the still-water state.

\exampletitle{exmp:two-layer-moving-equilibrium}{Moving-water equilibrium preservation}
In this example, we verify the well-balanced property for the moving-water equilibrium, namely the generalized moving-water steady state. We use the discontinuous bottom topography
\begin{equation}
    b(x)=\left\{
    \begin{array}
    {ll}-2, & \mathrm{if~}x<0, \\
    -1, & \text{otherwise,}
    \end{array}\right.
\label{eq:two-layer-moving-bottom}
\end{equation}
and the initial condition
\begin{equation}
    \begin{gathered}
    h_1(x,0)=(h_1)_{\mathrm{eq}}(x)=\left\{
    \begin{array}
    {ll}1.22373355048230, & \mathrm{if~}x<0, \\
    1.44970064153589, & \mathrm{otherwise}
    \end{array}\right. \\
    h_2(x,0)=(h_2)_{\mathrm{eq}}(x)=\left\{
    \begin{array}
    {ll}0.968329515483846, & \mathrm{if~}x<0, \\
    1.12439026921484, & \text{otherwise,}
    \end{array}\right. \\
    m_{1}(x,0)=(m_{1})_{\mathrm{eq}}(x)=12,\quad m_{2}(x,0)=(m_{2})_{\mathrm{eq}}(x)=10.
    \end{gathered}
\label{eq:two-layer-moving-equilibrium-initial-data}
\end{equation}
With $g=10.0$ and $r=0.98$, this discontinuous state satisfies $E_1\equiv 50$ and $E_2\equiv 55$, and therefore represents an exact moving-water equilibrium despite the jump in $h_1$, $h_2$, and $b$ at $x=0$. The numerical solution is computed on $[-1,1]$ up to $t=0.05$ with $N_x=100$ uniform cells.

\begin{table}[!htb]
\centering
\caption{Example~\ref{exmp:two-layer-moving-equilibrium}. $L_1$ and $L_\infty$ errors for the one-dimensional moving-water equilibrium state, $N_x=100$}
\label{tab:two-layer-moving-equilibrium}
\begin{tabular}{lllllll}
\toprule
error      & $h_1$        & $m_1$        & $h_2$        & $m_2$        & $E_1$        & $E_2$        \\
\midrule
$L_1$      & 2.22e-16 & 4.39e-15 & 1.06e-17 & 3.24e-16 & 2.47e-14 & 2.21e-15 \\
$L_\infty$ & 2.22e-16 & 3.55e-15 & 1.39e-17 & 3.03e-16 & 2.13e-14 & 2.06e-15 \\
\bottomrule
\end{tabular}
\end{table}

The round-off-level errors in Table~\ref{tab:two-layer-moving-equilibrium} confirm discrete preservation of the generalized two-layer moving-water equilibrium over discontinuous bottom topography, even on this relatively coarse mesh.

\exampletitle{exmp:two-layer-small-moving}{Small perturbation of moving-water equilibrium}

We finally consider a small perturbation of the moving-water equilibrium \eqref{eq:two-layer-moving-equilibrium-initial-data}. The bottom topography is still given by \eqref{eq:two-layer-moving-bottom}. A localized perturbation is added to the upper-layer depth while the lower-layer depth and both discharges are kept at their equilibrium values:
\begin{equation}
\begin{aligned}
    & h_1(x,0)=(h_1)_{\mathrm{eq}}(x)+\left\{
    \begin{array}
    {ll}0.001, & \mathrm{if~}x\in[-0.6,-0.5], \\
    0, & \text{otherwise,}
    \end{array}\right. \\
    & h_2(x,0)=(h_2)_{\mathrm{eq}}(x),\quad
    m_1(x,0)=(m_1)_{\mathrm{eq}}(x),\quad
    m_2(x,0)=(m_2)_{\mathrm{eq}}(x).
\end{aligned}
\label{eq:two-layer-moving-perturbation-initial-data}
\end{equation}
The solutions are shown at $t=0$, $0.02$, $0.05$, and $0.08$. We compute the PCDG-GHR solution with $N_x=200$ uniform cells and compare it with a reference solution computed with $N_x=1000$ uniform cells.

\begin{figure}[!htb]
\centering
\includegraphics[width=1\textwidth]{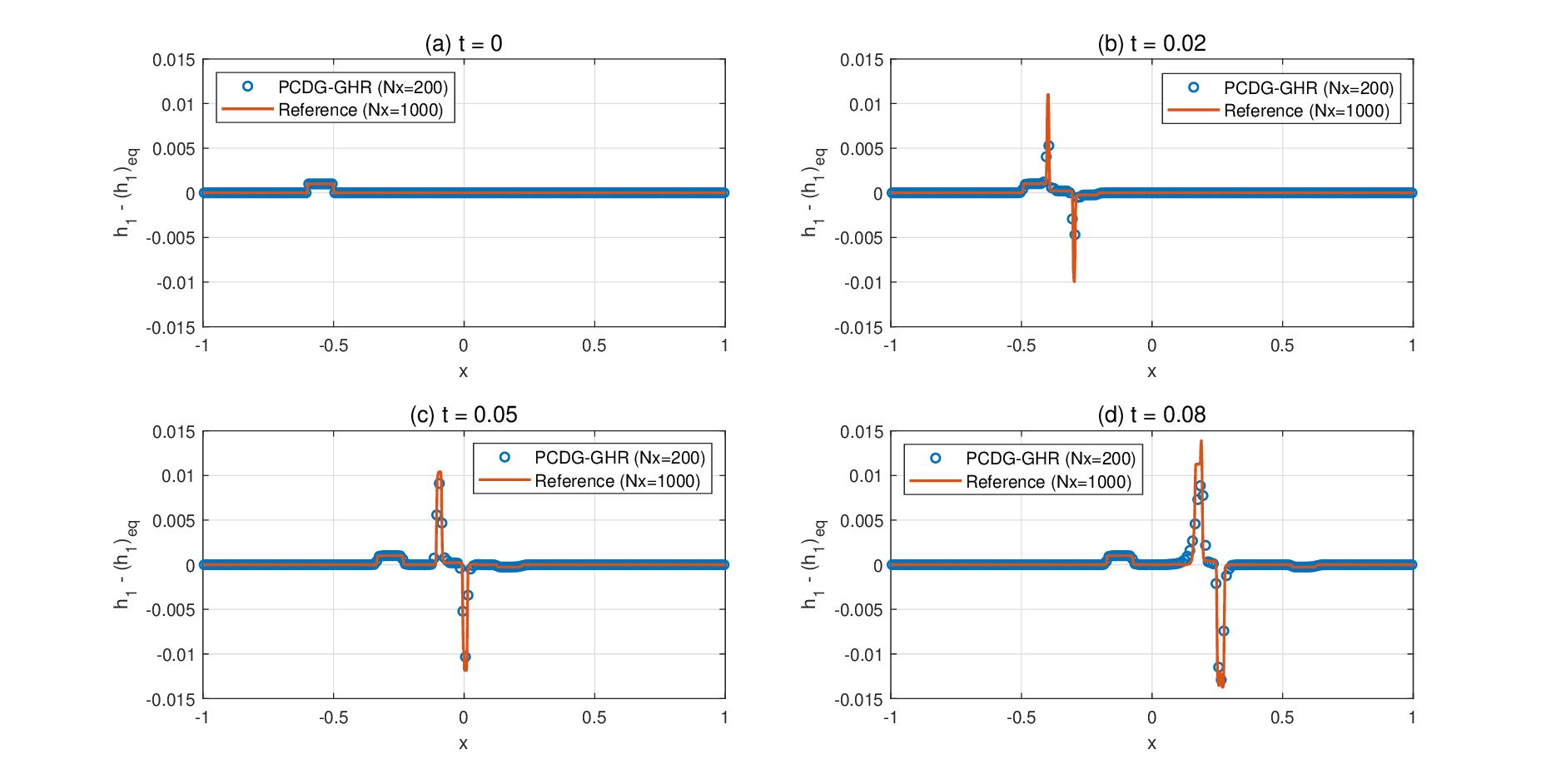}
\caption{Example~\ref{exmp:two-layer-small-moving}. Small perturbation of the two-layer moving-water equilibrium. The difference from the equilibrium upper-layer depth is shown at $t=0$, $0.02$, $0.05$, and $0.08$, with the PCDG-GHR solution computed on $N_x=200$ cells and the reference solution computed on $N_x=1000$ cells.}
\label{fig:two-layer-small-moving}
\end{figure}

Figure~\ref{fig:two-layer-small-moving} shows the evolution of the perturbation relative to the moving-water equilibrium. The coarse-mesh PCDG-GHR solution follows the refined reference solution at all displayed times and propagates the localized disturbance without producing spurious waves near the discontinuity of the bottom function. This confirms that the moving-water well-balanced treatment accurately resolves small perturbations of the moving-water equilibrium.

\exampletitle{exmp:two-layer-interface-propagation}{Interface propagation}

We use this example to examine the propagation of an internal interface in the two-layer system over a flat bottom,
\begin{equation}
b(x)=-1,
\label{eq:two-layer-interface-bottom}
\end{equation}
so that the test focuses on the non-conservative interlayer coupling rather than on a geometric source term. On the computational domain $[-1,1]$, the initial data are
\begin{equation}
(h_1,m_1,H_2,m_2)(x,0)=
\left\{
\begin{array}{ll}
(0.50,1.250,-0.50,1.250), & x<0.3,\\
(0.45,1.125,-0.45,1.375), & \text{otherwise,}
\end{array}
\right.
\label{eq:two-layer-interface-initial-data}
\end{equation}
so that the interface is initially located at $x=0.3$. We take $g=9.8$ and $r=0.98$, and compute the solution up to $t=0.1$. The PCDG-GHR solution is computed with $N_x=400$ uniform cells and compared with a refined reference solution with $N_x=4000$ cells.

\begin{figure}[!htb]
\centering
\includegraphics[width=1\textwidth]{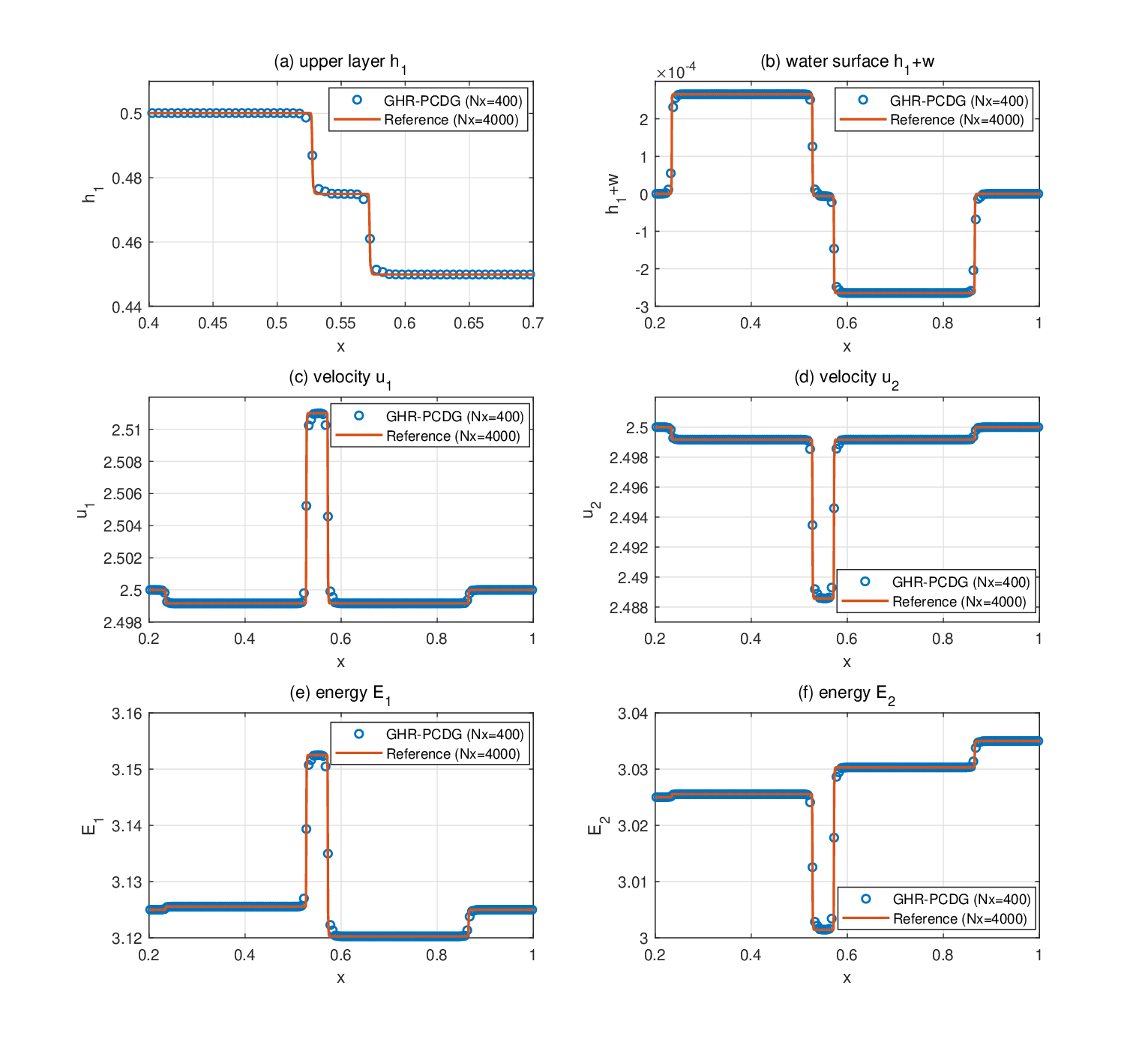}
\caption{Example~\ref{exmp:two-layer-interface-propagation}. Interface propagation for the two-layer shallow water equations. The PCDG-GHR solution with $N_x=400$ is compared with the reference solution with $N_x=4000$ for the upper-layer depth, the free surface $H_1$, the layer velocities, and the equilibrium variables.}
\label{fig:two-layer-interface-propagation}
\end{figure}

Figure~\ref{fig:two-layer-interface-propagation} shows that the method resolves the interface motion sharply while maintaining good agreement with the refined reference solution. The upper-layer depth develops an intermediate nearly flat state between the left and right states, and the velocities form corresponding plateau regions. The free surface $H_1$ and equilibrium variables remain well controlled across the propagating interface, indicating that the path-conservative GHR treatment captures the interlayer wave interaction without visible spurious oscillations.

\exampletitle{exmp:two-layer-riemann}{Riemann problem over discontinuous bottom topography}

We close the two-layer numerical tests with Riemann problems over a discontinuous bottom topography. The purpose is to compare the behavior of different path-conservative discretizations when both the solution and the prescribed bottom contain jumps. The bottom function is
\begin{equation}
b(x)=
\left\{
\begin{array}{ll}
-2, & x<0,\\
-1.5, & \text{otherwise,}
\end{array}
\right.
\label{eq:two-layer-riemann-bottom}
\end{equation}
on the computational domain $[-1,1]$. We take $g=10.0$ and $r=0.98$, and compute the solutions up to $t=0.1$. In the comparison, PCCU denotes the second-order path-conservative central-upwind scheme of Castro D\'iaz et al.~\cite{castro2019path}; the PCDG-still and PCDG-moving schemes are the equilibrium-preserving path-conservative DG schemes for the two-layer model in \cite{ZHANG2025}. These results are compared with the present PCDG-GHR method.

The first Riemann initial condition is
\begin{equation}
(h_1,m_1,h_2,m_2)(x,0)=
\left\{
\begin{array}{ll}
(1.0,1.5,1.0,1.0), & x<0,\\
(0.8,1.2,1.2,1.8), & \text{otherwise,}
\end{array}
\right.
\label{eq:two-layer-riemann-test1-initial-data}
\end{equation}
and the second one is
\begin{equation}
(h_1,m_1,h_2,m_2)(x,0)=
\left\{
\begin{array}{ll}
(1.5,1.1,1.0,1.4), & x<0,\\
(1.2,1.6,0.9,1.2), & \text{otherwise.}
\end{array}
\right.
\label{eq:two-layer-riemann-test2-initial-data}
\end{equation}

\begin{figure}[!htb]
\centering
\includegraphics[width=1\textwidth]{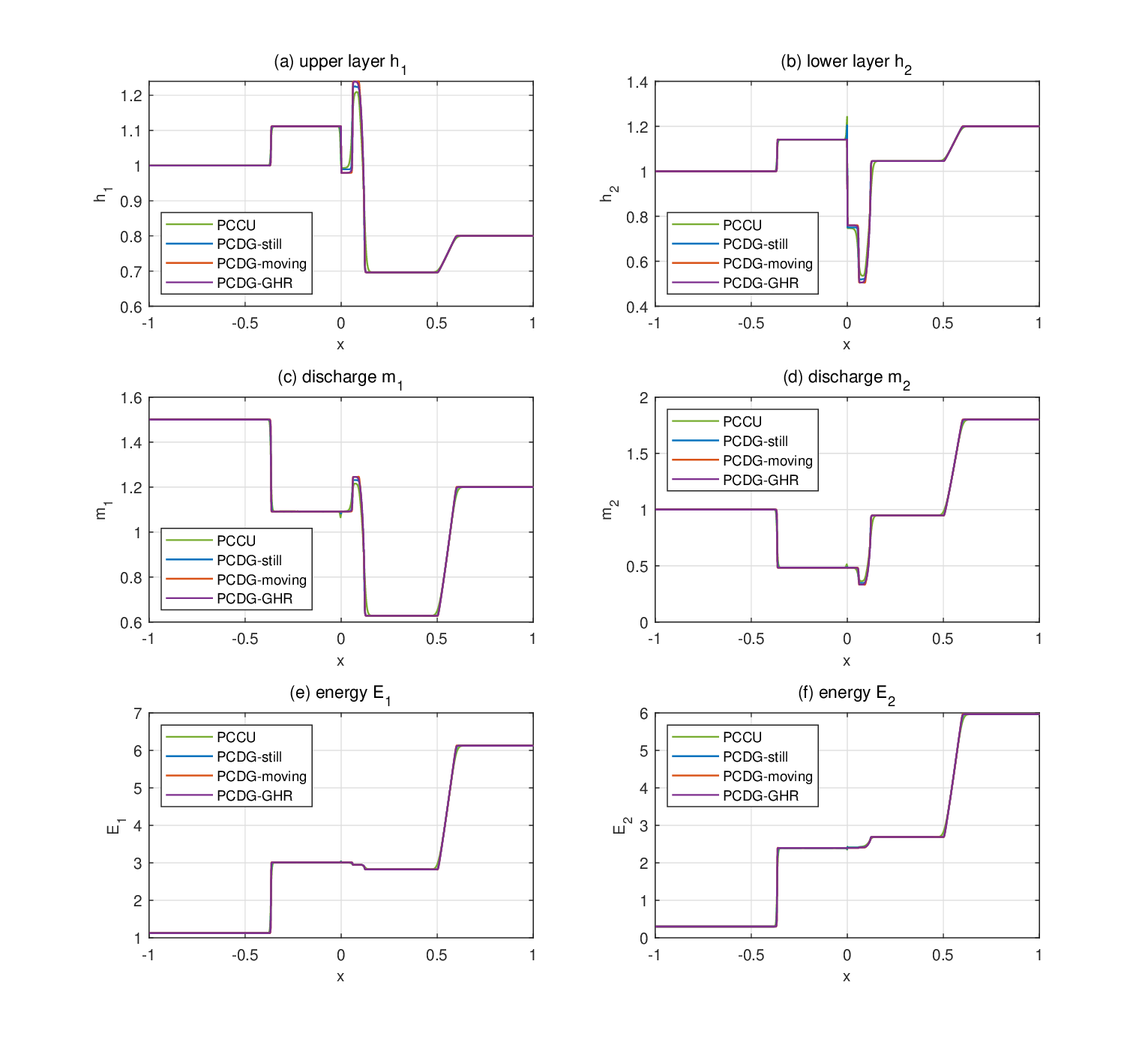}
\caption{Example~\ref{exmp:two-layer-riemann}. Riemann problem over discontinuous bottom topography for Test 1. The upper- and lower-layer depths, discharges, and equilibrium variables are compared among PCCU, PCDG-still, PCDG-moving, and PCDG-GHR.}
\label{fig:two-layer-riemann-test1}
\end{figure}

\begin{figure}[!htb]
\centering
\includegraphics[width=1\textwidth]{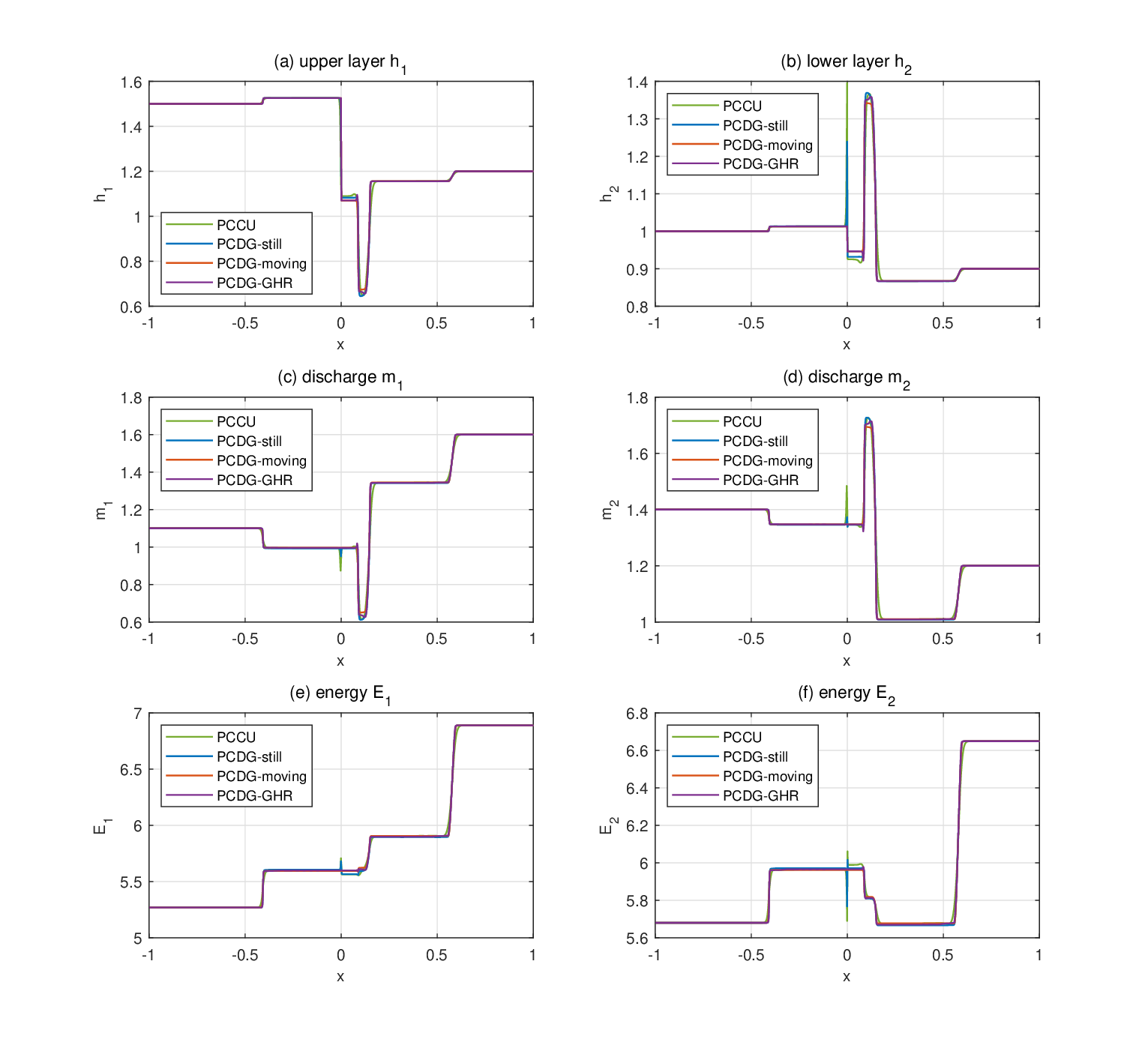}
\caption{Example~\ref{exmp:two-layer-riemann}. Riemann problem over discontinuous bottom topography for Test 2. The upper- and lower-layer depths, discharges, and equilibrium variables are compared among PCCU, PCDG-still, PCDG-moving, and PCDG-GHR.}
\label{fig:two-layer-riemann-test2}
\end{figure}

Figures~\ref{fig:two-layer-riemann-test1} and~\ref{fig:two-layer-riemann-test2} show that the treatment of the discontinuous bottom visibly affects the computed solution near $x=0$. The fixed jump in the bottom produces a stationary material contact, namely a zero-speed wave associated with the discontinuous prescribed data. In both tests, the PCCU and PCDG-still solutions produce spurious oscillatory peaks at this contact. By contrast, the PCDG-moving and PCDG-GHR solutions remain nonoscillatory and agree more closely with the reference solution. These comparisons indicate that preserving the moving-water equilibrium structure is important for resolving the stationary contact without generating spurious waves.

\subsection{Blood flow equations}

For the blood flow equations \eqref{eq:blood-flow-system}, we test high-order accuracy, preservation of blood-at-rest and moving-flow equilibria with smooth or discontinuous vessel properties, small perturbations of moving equilibria, and a Riemann problem with discontinuous vessel data. The examples are mainly based on \cite{pimentel2023high}. Unless stated otherwise, the spatial and temporal discretizations follow the settings above. We take $\rho=1050\;\mathrm{kg/m^3}$ throughout; each example specifies its tube-law exponents, vessel properties, domain, and initial and boundary data.

\exampletitle{exmp:blood-flow-accuracy}{Accuracy test}

To verify the high-order accuracy of the PCDG-GHR method for a smooth solution, we consider a vessel with smoothly varying properties. We take $\rho=1050\;\mathrm{kg/m^3}$, $(m_{\mathrm p},n_{\mathrm p})=(1/2,0)$, and the computational domain $[0,5]\;\mathrm{m}$. The prescribed vessel properties are
\begin{equation}
\begin{aligned}
K(x) &= 58725+100e^{-10(x-2.5)^2}\;\mathrm{Pa},\\
A_0(x) &= 5\times10^{-4}+10^{-4}e^{-10(x-2.5)^2}\;\mathrm{m^2},\\
p_e(x) &= 10000+100e^{-10(x-2.5)^2}\;\mathrm{Pa}.
\end{aligned}
\label{eq:blood-flow-accuracy-vessel-data}
\end{equation}
Let $A^{\mathrm{eq}}(x)$ denote the blood-at-rest equilibrium satisfying $A^{\mathrm{eq}}(0)=A_L=2.0456A_{0,L}$ and $q^{\mathrm{eq}}=0$. A smooth perturbation is added to the area,
\begin{equation}
A(x,0)=A^{\mathrm{eq}}(x)+10^{-6}e^{-40(x-1)^2}\;\mathrm{m^2},
\qquad q(x,0)=0.
\label{eq:blood-flow-accuracy-initial-data}
\end{equation}
We use an equilibrium-characteristic boundary treatment: at each boundary, the incoming characteristic information is prescribed by the equilibrium background state, whereas the outgoing information is taken from the interior solution. The simulations are run up to $t=0.5\;\mathrm{s}$. The numerical solution computed by the $P^2$ PCDG-GHR method on a refined mesh with $N_x=6400$ is used as the reference solution. The errors below are evaluated for the dimensionless variables $A/A_0$, $q/(5\times10^{-4}\;\mathrm{m^3/s})$, $u/(1\;\mathrm{m/s})$, and $\Gamma/(1050\;\mathrm{Pa})$.

\begin{table}[!htb]
\centering
\caption{Example~\ref{exmp:blood-flow-accuracy}. $L_1$ errors and numerical orders of accuracy with initial conditions \eqref{eq:blood-flow-accuracy-initial-data}, using $P^2$ polynomial.}
\label{tab:blood-flow-accuracy-test}
\begin{tabular}{lllllllll}
\toprule
\multirow{2}{*}{$N_x$} & \multicolumn{2}{l}{$A/A_0$} & \multicolumn{2}{l}{$q$} & \multicolumn{2}{l}{$u$} & \multicolumn{2}{l}{$\Gamma$} \\ \cline{2-9}
 & $L_1$ error & order & $L_1$ error & order & $L_1$ error & order & $L_1$ error & order \\
\midrule
 20  & 2.41e-04 & /    & 1.25e-04 & /    & 6.09e-05 & /    & 4.72e-03 & / \\
 40  & 3.29e-05 & 2.87 & 1.98e-05 & 2.65 & 9.68e-06 & 2.65 & 6.44e-04 & 2.87 \\
 80  & 3.78e-06 & 3.12 & 1.14e-06 & 4.12 & 5.57e-07 & 4.12 & 7.41e-05 & 3.12 \\
160  & 4.57e-07 & 3.05 & 7.01e-08 & 4.02 & 3.43e-08 & 4.02 & 8.94e-06 & 3.05 \\
320  & 5.67e-08 & 3.01 & 7.15e-09 & 3.29 & 3.49e-09 & 3.29 & 1.11e-06 & 3.01 \\
640  & 7.07e-09 & 3.00 & 8.61e-10 & 3.05 & 4.20e-10 & 3.05 & 1.38e-07 & 3.00 \\
\bottomrule
\end{tabular}
\end{table}

The results in Table~\ref{tab:blood-flow-accuracy-test} show the expected third-order accuracy for this smooth blood-flow test. The errors in $A/A_0$, $q$, $u$, and $\Gamma$ approach third-order convergence under mesh refinement. This verifies the designed accuracy of the $P^2$ PCDG-GHR method for the blood flow equations with smoothly varying vessel properties.

\exampletitle{exmp:blood-flow-rest-equilibrium}{Blood-at-rest equilibrium preservation with discontinuous vessel properties}

We consider a blood-at-rest equilibrium with discontinuous vessel properties. We take $\rho=1050\;\mathrm{kg/m^3}$, $(m_{\mathrm p},n_{\mathrm p})=(1/2,0)$, and the computational domain $[0,0.2]\;\mathrm{m}$, with the material interface at $x=0.1\;\mathrm{m}$. With $A_{0,\mathrm{ref}}=3.1353\times10^{-4}\;\mathrm{m^2}$ and $K_{\mathrm{ref}}=58725\;\mathrm{Pa}$, the prescribed vessel properties are
\begin{equation}
(K,A_0,p_e)(x)=
\begin{cases}
(K_{\mathrm{ref}},\,2A_{0,\mathrm{ref}},\,9999.15\;\mathrm{Pa}), & x<0.1,\\
(10K_{\mathrm{ref}},\,A_{0,\mathrm{ref}},\,11332.37\;\mathrm{Pa}), & x\ge 0.1.
\end{cases}
\label{eq:blood-flow-discontinuous-vessel-data}
\end{equation}
The initial state is
\begin{equation}
(A,u)(x,0)=
\begin{cases}
(1.0228A_{0,L},\,0), & x<0.1,\\
(0.9977A_{0,R},\,0), & x\ge 0.1.
\end{cases}
\label{eq:blood-flow-rest-initial-data}
\end{equation}
The tabulated right area is rounded; in the computation, the positive root of the blood-at-rest equilibrium relation is used to recover the area so that this relation is satisfied exactly. The PCDG-GHR method is run to $t=0.1\;\mathrm{s}$ on $N_x=100$ uniform cells. The errors below are evaluated in the physical variables relative to the exact equilibrium, with the $L_1$ norm evaluated over the physical coordinate. Thus, the $L_1$ errors in $A$, $u$, and $q$ have units $\mathrm{m^3}$, $\mathrm{m^2/s}$, and $\mathrm{m^4/s}$, respectively, while the corresponding $L_\infty$ errors have units $\mathrm{m^2}$, $\mathrm{m/s}$, and $\mathrm{m^3/s}$.

\begin{table}[!htb]
\centering
\caption{Example~\ref{exmp:blood-flow-rest-equilibrium}. $L_1$ and $L_\infty$ errors for the one-dimensional blood-at-rest equilibrium with discontinuous vessel properties at $t=0.1\;\mathrm{s}$, $N_x=100$}
\label{tab:blood-flow-rest-equilibrium}
\begin{tabular}{llll}
\toprule
error & $A$ & $u$ & $q$ \\
\midrule
$L_1$      & 8.42e-19 & 1.14e-13 & 5.18e-17 \\
$L_\infty$ & 1.50e-17 & 1.44e-12 & 6.29e-16 \\
\bottomrule
\end{tabular}
\end{table}

The errors in Table~\ref{tab:blood-flow-rest-equilibrium} remain at round-off level. This confirms that the PCDG-GHR method preserves the discontinuous blood-at-rest equilibrium without generating a numerical disturbance at the material interface.

\exampletitle{exmp:blood-flow-moving-equilibrium}{Moving-flow equilibrium preservation with discontinuous vessel properties}

To test the preservation of a moving-flow equilibrium with discontinuous vessel properties, we retain the density, domain, and interface location of Example~\ref{exmp:blood-flow-rest-equilibrium} and set $(m_{\mathrm p},n_{\mathrm p})=(10,-3/2)$. The vessel properties used in the computation are
\begin{equation}
(K,A_0,p_e)(x)=
\begin{cases}
(K_{\mathrm{ref}},\,2A_{0,\mathrm{ref}},\,9999.15\;\mathrm{Pa}), & x<0.1,\\
(10K_{\mathrm{ref}},\,A_{0,\mathrm{ref}},\,78002.20\;\mathrm{Pa}), & x\ge 0.1.
\end{cases}
\label{eq:blood-flow-moving-discontinuous-vessel-data}
\end{equation}
The initial state is
\begin{equation}
(A,u)(x,0)=
\begin{cases}
(1.0228A_{0,L},\,1\;\mathrm{m/s}), & x<0.1,\\
(3.109988\times10^{-4}\;\mathrm{m^2},\,2.06224886\;\mathrm{m/s}), & x\ge 0.1.
\end{cases}
\label{eq:blood-flow-moving-initial-data}
\end{equation}
The value $p_{e,R}=11332.37\;\mathrm{Pa}$ printed with these states in \cite{pimentel2023high} does not satisfy the stationary relation. We therefore use $p_{e,R}=78002.20\;\mathrm{Pa}$, obtained by enforcing common values of $q$ and $\Gamma$ across the material interface. The tabulated right state is used to identify the subcritical branch, and the area is recovered from the equilibrium relation in the computation. The PCDG-GHR solution is computed on $N_x=100$ uniform cells up to $t=0.1\;\mathrm{s}$. The errors use the same physical-variable convention as in the preceding test.

\begin{table}[!htb]
\centering
\caption{Example~\ref{exmp:blood-flow-moving-equilibrium}. $L_1$ and $L_\infty$ errors for the one-dimensional moving-flow equilibrium with discontinuous vessel properties at $t=0.1\;\mathrm{s}$, $N_x=100$}
\label{tab:blood-flow-moving-equilibrium}
\begin{tabular}{llll}
\toprule
error & $A$ & $u$ & $q$ \\
\midrule
$L_1$      & 2.72e-18 & 5.50e-13 & 2.73e-16 \\
$L_\infty$ & 2.85e-17 & 4.38e-12 & 2.61e-15 \\
\bottomrule
\end{tabular}
\end{table}

Table~\ref{tab:blood-flow-moving-equilibrium} shows round-off-level errors in the reported physical variables. This confirms that the PCDG-GHR method preserves the subcritical moving-flow equilibrium across simultaneous jumps in the vessel stiffness, unloaded area, and external pressure.

\exampletitle{exmp:blood-flow-small-discontinuous}{Small perturbation of a discontinuous moving-flow equilibrium}

We next perturb the discontinuous subcritical moving equilibrium of Example~\ref{exmp:blood-flow-moving-equilibrium}. Let $A^{\mathrm{eq}}$ and $q^{\mathrm{eq}}$ denote the recovered equilibrium area and its constant discharge, respectively. The initial perturbation is
\begin{equation}
A(x,0)=A^{\mathrm{eq}}(x)+10^{-5}e^{-20000(x-0.05)^2}\;\mathrm{m^2},
\qquad q(x,0)=q^{\mathrm{eq}}.
\label{eq:blood-flow-small-discontinuous-perturbation}
\end{equation}

The $P^2$ PCDG-GHR method is run on $N_x=100$ uniform cells with the same equilibrium-characteristic boundary treatment and no limiter. A PCDG-GHR solution computed with the same settings on $N_x=1000$ cells is used as the reference solution. Figures~\ref{fig:blood-flow-small-discontinuous} and~\ref{fig:blood-flow-small-discontinuous-velocity} compare $A/A_0$ and $u$, respectively, at $t=0.001$, $0.002$, and $0.1\;\mathrm{s}$.

\begin{figure}[!htb]
\centering
\includegraphics[width=1\textwidth]{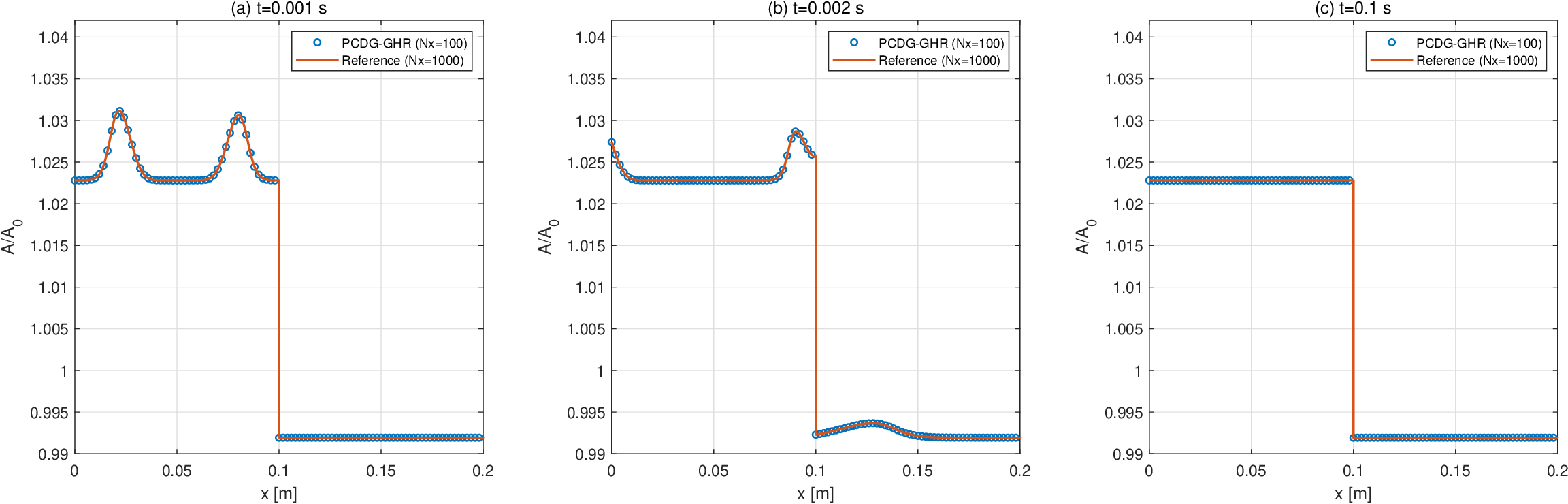}
\caption{Example~\ref{exmp:blood-flow-small-discontinuous}. Evolution of $A/A_0$ at $t=0.001$, $0.002$, and $0.1\;\mathrm{s}$. The $P^2$ PCDG-GHR solution with $N_x=100$ is compared with the $P^2$ PCDG-GHR reference solution with $N_x=1000$.}
\label{fig:blood-flow-small-discontinuous}
\end{figure}

\begin{figure}[!htb]
\centering
\includegraphics[width=1\textwidth]{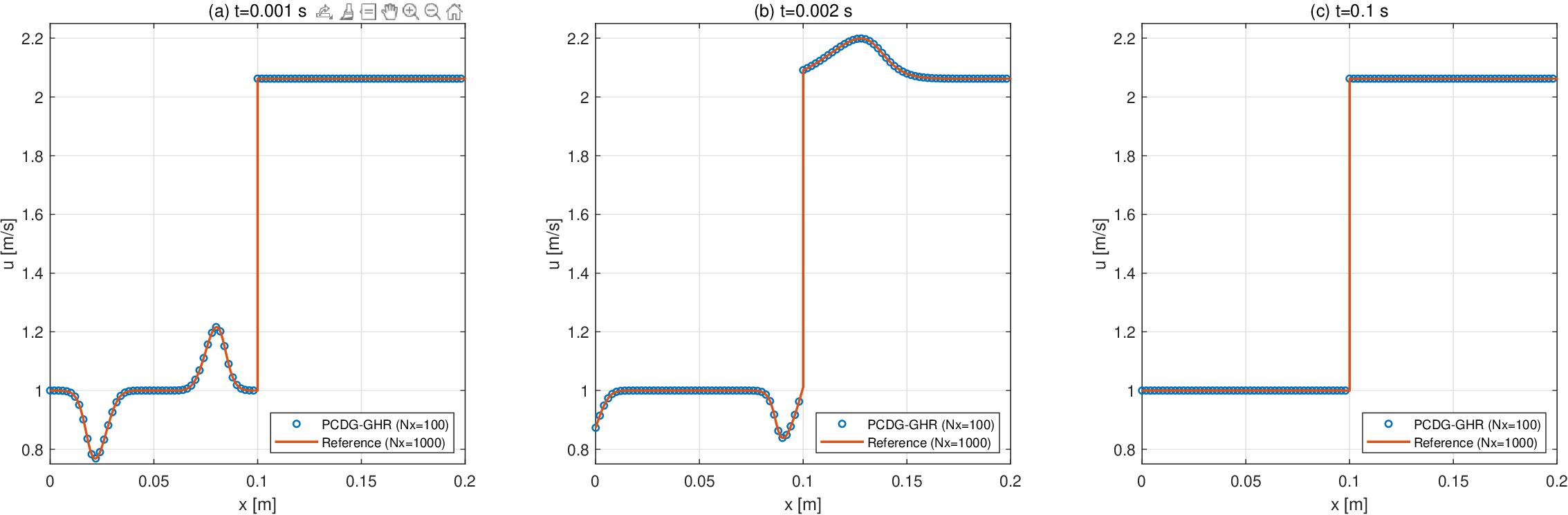}
\caption{Example~\ref{exmp:blood-flow-small-discontinuous}. Evolution of the velocity $u$ at $t=0.001$, $0.002$, and $0.1\;\mathrm{s}$. The $P^2$ PCDG-GHR solution with $N_x=100$ is compared with the $P^2$ PCDG-GHR reference solution with $N_x=1000$.}
\label{fig:blood-flow-small-discontinuous-velocity}
\end{figure}

The perturbation separates into two waves and interacts with the discontinuity in the vessel properties. Figures~\ref{fig:blood-flow-small-discontinuous} and~\ref{fig:blood-flow-small-discontinuous-velocity} show that the area ratio and velocity agree well with the refined reference solution at all displayed times. No visible spurious oscillations appear near the stationary material contact. After the perturbation leaves the domain, the solution returns to the underlying moving equilibrium.

\exampletitle{exmp:blood-flow-smooth-equilibrium}{Smooth moving-flow equilibrium preservation}

We verify preservation of a smooth moving-flow equilibrium using the density, tube law, domain, and vessel properties of Example~\ref{exmp:blood-flow-accuracy}. Let $(A^{\mathrm{eq}},q^{\mathrm{eq}})$ denote the subcritical stationary solution satisfying $A^{\mathrm{eq}}(0)=A_L=2.0456A_{0,L}$ and $u^{\mathrm{eq}}(0)=1\;\mathrm{m/s}$. For each $x\in[0,5]\;\mathrm{m}$, the area is recovered on the subcritical branch from the constant equilibrium variables and the local values of $(K,A_0,p_e)$.

The PCDG-GHR method is applied on $N_x=100$ uniform cells with equilibrium boundary conditions. The limiter is turned off, and the solution is evolved to $t=0.1\;\mathrm{s}$. Since the smooth equilibrium is not represented exactly in the $P^2$ polynomial space, Table~\ref{tab:blood-flow-smooth-equilibrium} reports the errors in $A$, $u$, and $q$ relative to the initial discrete equilibrium, using the same physical-coordinate $L_1$ and $L_\infty$ norms as above.

\begin{table}[!htb]
\centering
\caption{Example~\ref{exmp:blood-flow-smooth-equilibrium}. $L_1$ and $L_\infty$ errors for the one-dimensional smooth moving-flow equilibrium at $t=0.1\;\mathrm{s}$, $N_x=100$}
\label{tab:blood-flow-smooth-equilibrium}
\begin{tabular}{llll}
\toprule
error & $A$ & $u$ & $q$ \\
\midrule
$L_1$      & 1.31e-18 & 2.64e-14 & 2.77e-17 \\
$L_\infty$ & 1.55e-18 & 2.99e-14 & 3.06e-17 \\
\bottomrule
\end{tabular}
\end{table}

All reported errors remain at round-off level. These results confirm that the PCDG-GHR method preserves the smooth subcritical moving-flow equilibrium with spatially varying vessel properties.

\exampletitle{exmp:blood-flow-small-smooth}{Small perturbation of a smooth moving-flow equilibrium}

We next perturb the smooth moving-flow equilibrium of Example~\ref{exmp:blood-flow-smooth-equilibrium}. The initial data are
\begin{equation}
A(x,0)=A^{\mathrm{eq}}(x)+10^{-7}e^{-40(x-1)^2}\;\mathrm{m^2},
\qquad q(x,0)=q^{\mathrm{eq}},
\label{eq:blood-flow-small-smooth-perturbation}
\end{equation}
The $P^2$ PCDG-GHR method is applied on $N_x=100$ uniform cells with the same equilibrium-characteristic boundary treatment and no limiter. The reference solution is computed by the same method on $N_x=1000$ cells.

\begin{figure}[!htb]
\centering
\includegraphics[width=1\textwidth]{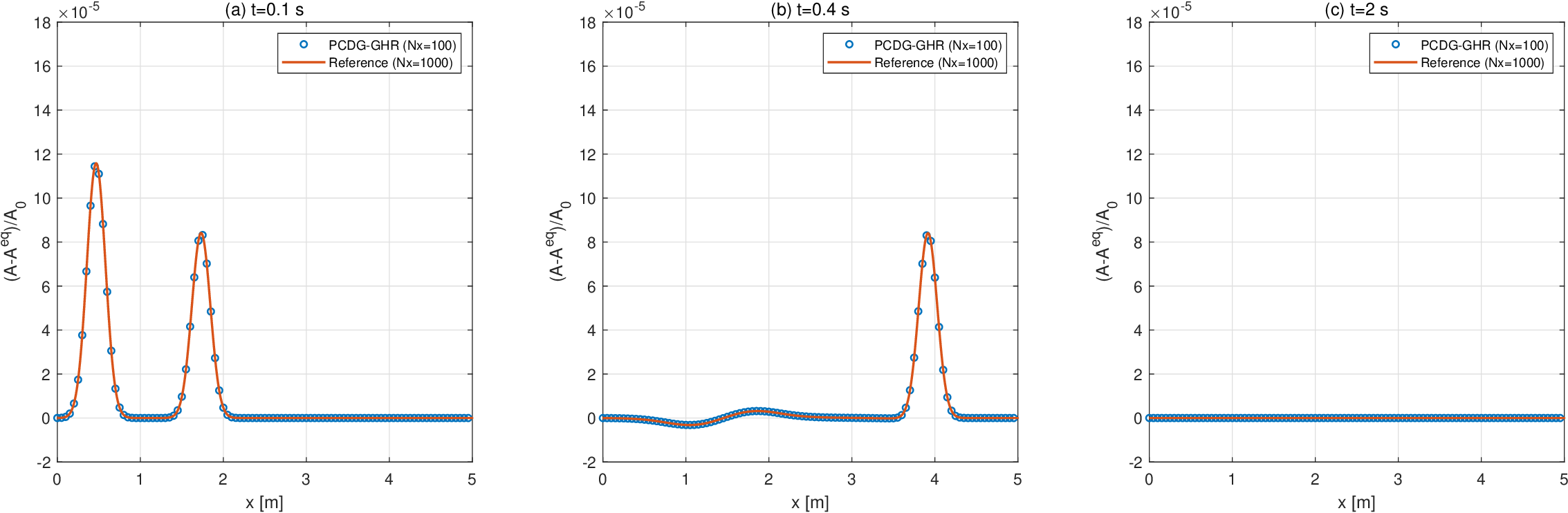}
\caption{Example~\ref{exmp:blood-flow-small-smooth}. Evolution of the perturbation $(A-A^{\mathrm{eq}})/A_0$ at $t=0.1$, $0.4$, and $2.0\;\mathrm{s}$. The $P^2$ PCDG-GHR solution with $N_x=100$ is compared with the $P^2$ PCDG-GHR reference solution with $N_x=1000$.}
\label{fig:blood-flow-small-smooth}
\end{figure}

\begin{figure}[!htb]
\centering
\includegraphics[width=1\textwidth]{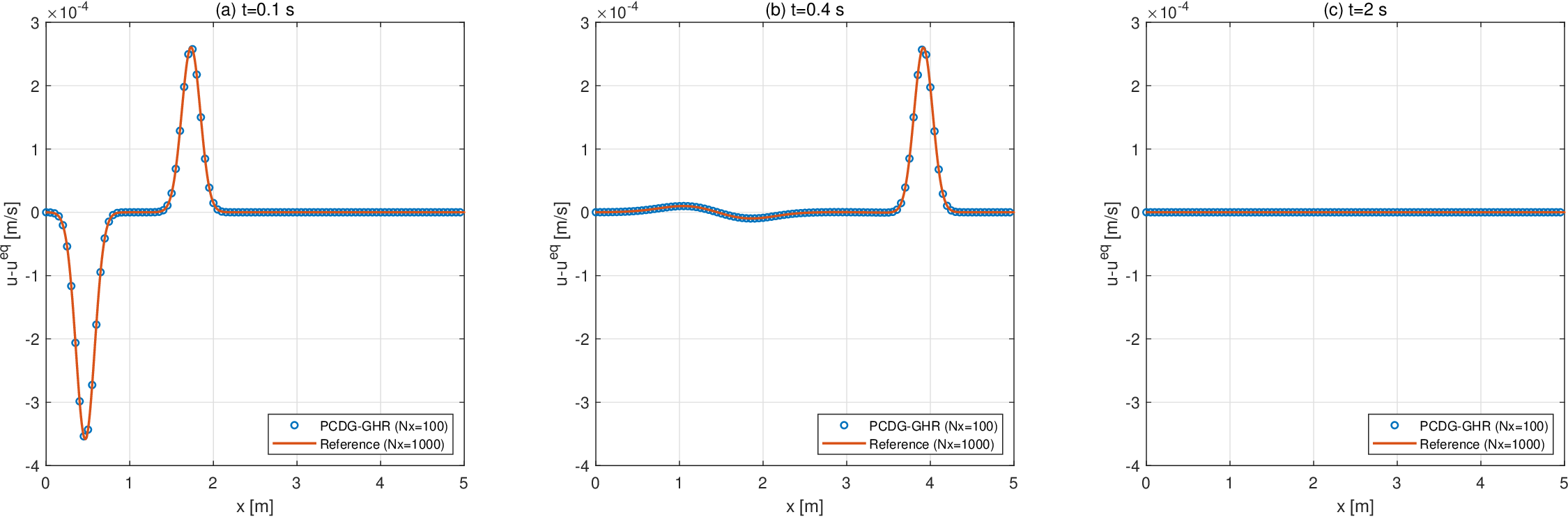}
\caption{Example~\ref{exmp:blood-flow-small-smooth}. Evolution of the velocity perturbation $u-u^{\mathrm{eq}}$ at $t=0.1$, $0.4$, and $2.0\;\mathrm{s}$. The $P^2$ PCDG-GHR solution with $N_x=100$ is compared with the $P^2$ PCDG-GHR reference solution with $N_x=1000$.}
\label{fig:blood-flow-small-smooth-velocity}
\end{figure}

Figures~\ref{fig:blood-flow-small-smooth} and~\ref{fig:blood-flow-small-smooth-velocity} show that the method captures the propagation and deformation of the small perturbation in both area and velocity. The solution on $N_x=100$ cells agrees well with the refined reference solution at all displayed times. After the waves leave the domain, the solution returns to the smooth moving equilibrium without a visible residual wave.

\exampletitle{exmp:blood-flow-riemann}{Riemann problem with discontinuous vessel properties}

Finally, we consider a Riemann problem in which both physical variables and all three prescribed vessel properties are discontinuous. We take $\rho=1050\;\mathrm{kg/m^3}$, $(m_{\mathrm p},n_{\mathrm p})=(10,-3/2)$, and place the initial discontinuity at $x=0.25\;\mathrm{m}$ in the domain $[0,0.5]\;\mathrm{m}$. The initial data are
\begin{equation}
\begin{array}{c|ccccc}
 & A\,[\mathrm{m^2}] & u\,[\mathrm{m/s}] & K\,[\mathrm{Pa}]
 & A_0\,[\mathrm{m^2}] & p_e\,[\mathrm{Pa}] \\ \hline
x<0.25
 & 1.65\times10^{-4} & 1.8 & 50 & 1.1\times10^{-4} & 66.661 \\
x\ge 0.25
 & 1.32\times10^{-4} & 2.0 & 500 & 1.2\times10^{-4} & 6.6661
\end{array}
\label{eq:blood-flow-riemann-initial-data}
\end{equation}
where $q=Au$. Free boundary conditions are imposed. The $P^2$ PCDG-GHR and standard PCDG solutions are computed with $N_x=400$ uniform cells up to $t=0.03\;\mathrm{s}$. Here the standard PCDG method removes the GHR treatment and the reference equilibrium decomposition, while all other numerical settings remain the same as those of the PCDG-GHR method. The TVB and area-positivity limiters are activated for both computations. A same-path first-order GHR solution computed on $N_x=3200$ cells is used as the numerical reference solution; here same-path means that it uses the same seven-segment augmented path as PCDG-GHR.

\begin{figure}[!htb]
\centering
\includegraphics[width=1\textwidth]{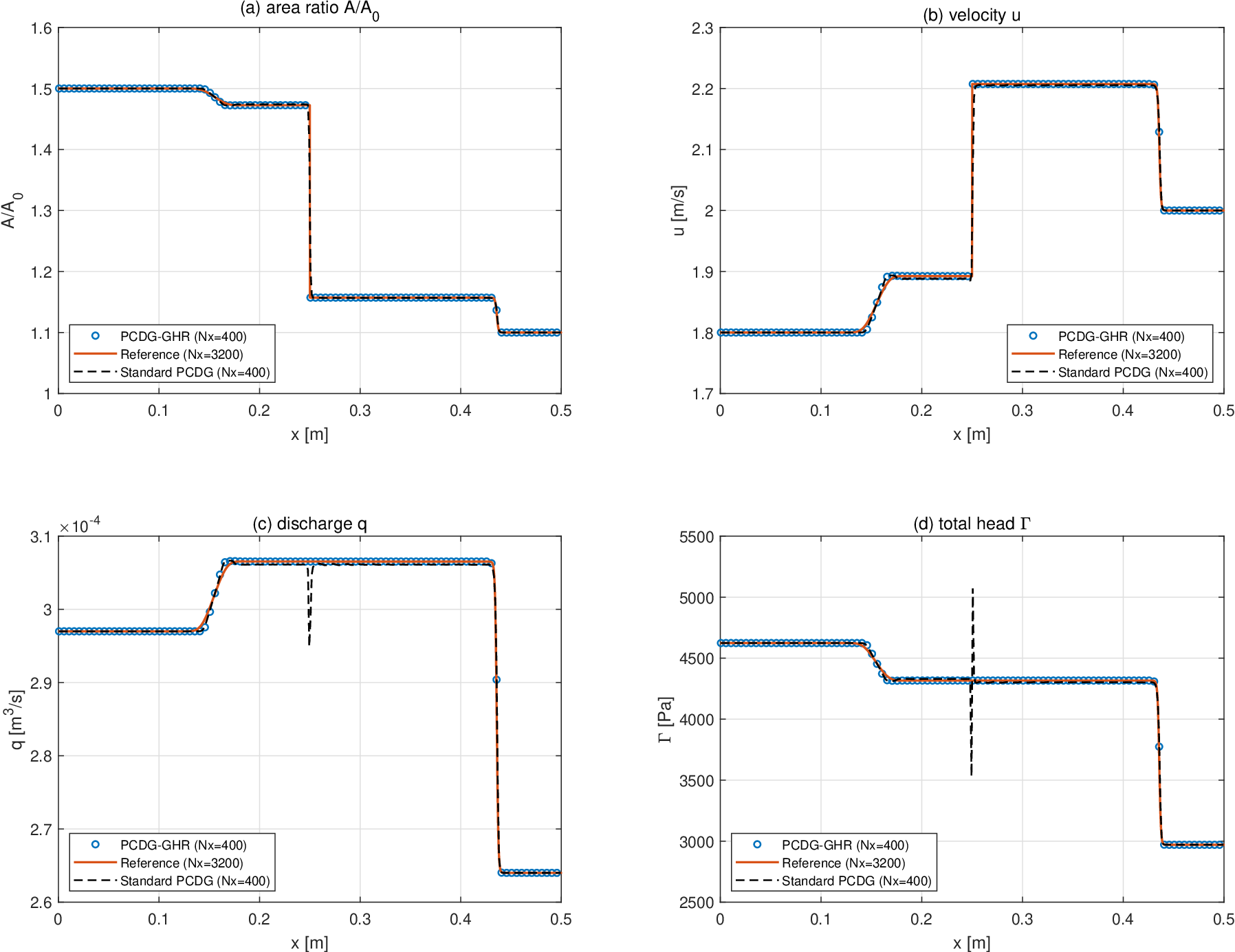}
\caption{Example~\ref{exmp:blood-flow-riemann}. Riemann problem with discontinuous vessel properties at $t=0.03\;\mathrm{s}$. The $P^2$ PCDG-GHR solution with $N_x=400$ is compared with the same-path P0 reference solution with $N_x=3200$ and the standard $P^2$ PCDG solution with $N_x=400$.}
\label{fig:blood-flow-riemann-test15}
\end{figure}

The solution consists of a left-moving rarefaction, a stationary material contact, and a right-moving shock. Figure~\ref{fig:blood-flow-riemann-test15} shows that the PCDG-GHR method captures these waves and agrees well with the reference solution for $A/A_0$, $u$, $q$, and $\Gamma$. In particular, the right-moving shock is sharply resolved, and no visible spurious oscillations appear at the stationary contact. By contrast, the standard PCDG solution develops nonphysical deviations near the discontinuity in the vessel properties. This comparison demonstrates the importance of the GHR path-conservative interface treatment for resolving the material contact.

\begin{remark}
The full well-balanced formulation with the reference equilibrium decomposition is most useful when the solution is an exact equilibrium or a small perturbation of a prescribed equilibrium family. In such regimes, local recovery removes the cell-interior discrete imbalance caused by polynomial projection errors. In the discontinuity-dominated Riemann tests for the two-layer shallow water and blood flow equations, the local states may span different admissible branches or characteristic regimes and need not remain close to a single equilibrium family. In both tests, we retain the path-conservative GHR interface treatment and omit the reference equilibrium decomposition in the volume residual. The decomposition is designed to remove discrete imbalances caused by projecting the equilibrium into the DG polynomial space; it does not determine which weak shock solution is selected far from equilibrium. The Riemann comparisons therefore assess primarily the GHR interface treatment together with the limiters used in the computations.
\end{remark}

\section{Conclusion}
\label{se:con}

We have developed a high-order well-balanced path-conservative discontinuous Galerkin method for one-dimensional balance laws with non-conservative products. The method combines generalized hydrostatic reconstruction at cell interfaces with a cellwise reference-equilibrium decomposition. The interface reconstruction couples effective numerical fluxes with Dal Maso-LeFloch-Murat path integrals, whereas the equilibrium decomposition cancels projection-induced residuals in the volume discretization. Under the stated recovery, admissibility, path-consistency, interface-matching, and boundary assumptions, the spatial discretization is path conservative, and prescribed rest and moving equilibria are exactly preserved by both the semi-discrete scheme and its fully discrete SSP-RK approximation.

Applications to the two-layer shallow water equations and a blood-flow model confirm the designed high-order accuracy and round-off-level preservation of smooth and discontinuous equilibria. The method also accurately resolves small perturbations around moving equilibria. In the discontinuity-dominated Riemann tests, the GHR interface treatment—used without the reference-equilibrium volume decomposition—captures shocks and stationary material contacts without visible spurious oscillations. The present formulation is restricted to one-dimensional systems and locally recoverable equilibrium branches. Future work will address multidimensional extensions, multilayer shallow water systems, and blood-flow networks with junctions and bifurcations.

\section*{Acknowledgments}
The research of Yinhua Xia was partially supported by the Anhui Provincial Natural Science Foundation (grant No.~2408085J004) and the National Natural Science Foundation of China (NSFC grant No.~12671438). The work of Yan Xu was partially supported by the National Natural Science Foundation of China (NSFC grant No. 12471347),  and Laoshan Laboratory (No. LSKJ202300305).

\appendix

\section{Fluctuation form and path conservation of the standard PCDG scheme}
\label{app:standard-pcdg-path-conservation}

Following \cite{pares2006numerical}, the semi-discrete path-conservative DG scheme can be written in the abstract fluctuation form
\begin{equation}
\begin{aligned}
&\int_{I_j} (\boldsymbol{u}_h)_t \cdot \boldsymbol{\phi} \, dx
+ \int_{I_j} \mathcal{A}(\boldsymbol{u}_h)(\boldsymbol{u}_h)_x \cdot \boldsymbol{\phi} \, dx
- \int_{I_j} \boldsymbol{s}(\boldsymbol{u}_h,\boldsymbol{\sigma})\cdot \boldsymbol{\phi} \, dx \\
&\qquad
+ \boldsymbol{D}^+_{j-\frac{1}{2}}\cdot \boldsymbol{\phi}^+_{j-\frac{1}{2}}
+ \boldsymbol{D}^-_{j+\frac{1}{2}}\cdot \boldsymbol{\phi}^-_{j+\frac{1}{2}} = 0,
\end{aligned}
\label{eq:abstract-fluctuation-form}
\end{equation}
where $\mathcal{A}(\boldsymbol{u})=\partial \boldsymbol{f}/\partial \boldsymbol{u} + G(\boldsymbol{u})$. The fluctuations satisfy
\begin{equation}
\boldsymbol{D}^\pm(\boldsymbol{u},\boldsymbol{u})=0,
\qquad \forall\, \boldsymbol{u}\in\Omega,
\label{eq:fluctuation-consistency}
\end{equation}
and
\begin{equation}
\boldsymbol{D}^-(\boldsymbol{u}^-,\boldsymbol{u}^+) + \boldsymbol{D}^+(\boldsymbol{u}^-,\boldsymbol{u}^+)
=
\int_0^1 \mathcal{A}(\boldsymbol{\Phi}(s;\boldsymbol{u}^-,\boldsymbol{u}^+))\boldsymbol{\Phi}_s(s;\boldsymbol{u}^-,\boldsymbol{u}^+) \, ds.
\label{eq:fluctuation-path-conservation}
\end{equation}
Expanding $\mathcal{A}$ in \eqref{eq:abstract-fluctuation-form} gives
\begin{align}
\frac{d}{dt}\int_{I_j} \boldsymbol{u}_h \cdot \boldsymbol{\phi} \, dx
&=
\int_{I_j} \boldsymbol{f}(\boldsymbol{u}_h)\cdot \boldsymbol{\phi}_x \, dx
- \int_{I_j} G(\boldsymbol{u}_h)(\boldsymbol{u}_h)_x \cdot \boldsymbol{\phi} \, dx
+ \int_{I_j} \boldsymbol{s}(\boldsymbol{u}_h,\boldsymbol{\sigma})\cdot \boldsymbol{\phi} \, dx \nonumber\\
&\quad
- \bigl(\boldsymbol{f}(\boldsymbol{u}^-_{j+\frac{1}{2}})+\boldsymbol{D}^-_{j+\frac{1}{2}}\bigr)\cdot \boldsymbol{\phi}^-_{j+\frac{1}{2}}
+ \bigl(\boldsymbol{f}(\boldsymbol{u}^+_{j-\frac{1}{2}})-\boldsymbol{D}^+_{j-\frac{1}{2}}\bigr)\cdot \boldsymbol{\phi}^+_{j-\frac{1}{2}}.
\label{eq:expanded-fluctuation-form}
\end{align}
Define
\begin{equation}
\boldsymbol{D}^-_{j+\frac{1}{2}}
=
\widehat{\boldsymbol{f}}_{j+\frac{1}{2}} - \boldsymbol{f}(\boldsymbol{u}^-_{j+\frac{1}{2}})
+ \frac{1}{2}\boldsymbol{g}_{\boldsymbol{\Phi},j+\frac{1}{2}},
\qquad
\boldsymbol{D}^+_{j+\frac{1}{2}}
=
\boldsymbol{f}(\boldsymbol{u}^+_{j+\frac{1}{2}}) - \widehat{\boldsymbol{f}}_{j+\frac{1}{2}}
+ \frac{1}{2}\boldsymbol{g}_{\boldsymbol{\Phi},j+\frac{1}{2}}.
\label{eq:fluctuation-definitions}
\end{equation}
Then
\begin{equation}
\boldsymbol{f}(\boldsymbol{u}^-_{j+\frac{1}{2}})+\boldsymbol{D}^-_{j+\frac{1}{2}}
=
\widehat{\boldsymbol{f}}_{j+\frac{1}{2}}+\frac{1}{2}\boldsymbol{g}_{\boldsymbol{\Phi},j+\frac{1}{2}},
\qquad
\boldsymbol{f}(\boldsymbol{u}^+_{j-\frac{1}{2}})-\boldsymbol{D}^+_{j-\frac{1}{2}}
=
\widehat{\boldsymbol{f}}_{j-\frac{1}{2}}-\frac{1}{2}\boldsymbol{g}_{\boldsymbol{\Phi},j-\frac{1}{2}}.
\label{eq:interface-fluctuation-relations}
\end{equation}
Substitution into \eqref{eq:expanded-fluctuation-form} yields \eqref{eq:standard-pcdg-form}. Consistency of the numerical flux and the identity $\boldsymbol{g}_{\boldsymbol{\Phi}}=0$ for coincident states give \eqref{eq:fluctuation-consistency}. Moreover,
\begin{align}
\boldsymbol{D}^-_{j+\frac{1}{2}}+\boldsymbol{D}^+_{j+\frac{1}{2}}
&=
\boldsymbol{f}(\boldsymbol{u}^+_{j+\frac{1}{2}})-\boldsymbol{f}(\boldsymbol{u}^-_{j+\frac{1}{2}})
+ \boldsymbol{g}_{\boldsymbol{\Phi},j+\frac{1}{2}} \nonumber\\
&=
\int_0^1 \mathcal{A}(\boldsymbol{\Phi}_{j+\frac{1}{2}}(s))(\boldsymbol{\Phi}_{j+\frac{1}{2}})_s(s)\,ds,
\label{eq:path-conservation-proof}
\end{align}
which proves the path-conservation property of the standard semi-discrete PCDG formulation.

\bibliographystyle{abbrv}
\bibliography{sample}

\end{document}